\documentclass[11pt,a4paper]{article}
\pdfoutput=1

\usepackage[T1]{fontenc}
\usepackage[utf8]{inputenc}
\usepackage{amsmath,amssymb,amsthm,mathtools}
\usepackage{graphicx}
\usepackage{tikz}
\usepackage{pgfplots}
\usepackage{authblk}
\usepackage{enumitem}
\usepackage{url}
\usepackage[hidelinks]{hyperref}
\pgfplotsset{compat=1.18}
\usetikzlibrary{arrows}

\allowdisplaybreaks[1]

\newtheorem{definition}{Definition}[section]
\newtheorem{theorem}[definition]{Theorem}
\newtheorem{lemma}[definition]{Lemma}

\newtheorem{conjecture}[definition]{Conjecture}
\newtheorem{claim}[definition]{Claim}

\theoremstyle{remark}

\begin{document}

\title{Books versus Triangles near the $n/6$ Threshold}
\date{\today}
\author[1]{Kaizhe Chen}
\author[2]{Jie Ma}
\author[2]{Tianhen Wang}
\affil[1]{School of the Gifted Young, University of Science and Technology of China, Hefei, Anhui, 230026, China}
\affil[2]{School of Mathematical Sciences, University of Science and Technology of China, Hefei, Anhui, 230026, China}
\maketitle

\begin{abstract}
The book number \(b(G)\) of a graph \(G\) is the maximum number of triangles sharing a common edge. A strengthening of Mantel's theorem due to Rademacher states that every \(n\)-vertex graph with more than \(\lfloor n^2/4\rfloor\) edges contains at least \(\lfloor n/2\rfloor\) triangles. Another strengthening, initiated by Erd\H{o}s, asserts that every such graph $G$ satisfies \(b(G)\ge n/6\).
Motivated by these results, Mubayi studied the tradeoff between the total number of triangles and the book number in such graphs, and asymptotically resolved the problem when \(n/4\le b(G)\le n/2\). 
Conlon, Fox, and Sudakov conjectured that, for \(n/6\le b< n/4\), every \(n\)-vertex graph with at least \(\lfloor n^2/4\rfloor\) edges and book number at most \(b\), other than the balanced complete bipartite graph, has at least \(b^2(n-4b)\) triangles, with equality only for the blow-up \(S_{b,n}\) of the \(3\)-prism.
They proved the conjecture when \(b\) lies in an interval with endpoint \(n/4\), and also at the endpoint \(b=n/6\), where they asked whether it remains valid in an interval containing this endpoint.
In this paper, we answer this question affirmatively. We show that there exists a constant \(\varepsilon>0\) such that the conjecture holds for all \(n/6\le b\le (1/6+\varepsilon)n\). Our proof first establishes a stability theorem showing that every extremal graph is close to a blow-up of the \(3\)-prism, and then uses a detailed parameter analysis to force the exact six-partite structure.
\end{abstract}

\section{Introduction}
Mantel's theorem~\cite{Mantel07}, one of the cornerstones of extremal graph theory, states that the balanced complete bipartite graph is the unique triangle-free graph on $n$ vertices with $\lfloor n^2/4\rfloor$ edges. Thus an $n$-vertex graph $G$ with at least $\lfloor n^2/4\rfloor+1$ edges must contain a triangle. Two classical strengthenings investigate how robust this conclusion is. 
A {\it book} of size \(k\) in a graph is a collection of \(k\) triangles sharing a common edge. The {\it book number} of a graph \(G\), denoted \(b(G)\), is the largest size of a book in \(G\), and \(t(G)\) denotes the number of triangles in \(G\). 
Rademacher, in unpublished work later revisited by Erd\H{o}s~\cite{Erdos62}, proved that such a graph $G$ satisfies $t(G)\geq \lfloor n/2\rfloor$; this is sharp, as shown by adding one edge to the larger class of a balanced complete bipartite graph. 
A different strengthening, initiated by Erd\H{o}s and proved by Edwards and independently by Khad\v ziivanov and Nikiforov~\cite{KN79}, states that such a graph \(G\) satisfies \(b(G)\geq n/6\).
This is also sharp, as witnessed by the balanced blow-up of the 3-prism (i.e., \(S_{n/6,n}\), as illustrated in Figure~\ref{Fig-S_bn}).

The extremal graph for Rademacher's theorem has only \(\lfloor n/2\rfloor\) triangles, but all of them form one large book, while the extremal graph for the second result has book number exactly \(\lceil n/6 \rceil\), but \(\Omega(n^3)\) triangles. Motivated by these observations, Mubayi~\cite{Mubayi12} studied the interplay between the global parameter \(t(G)\) and the local parameter \(b(G)\) for graphs with \(n\) vertices and at least \(\lfloor n^2/4\rfloor\) edges. He proved the following asymptotically sharp result when the book number lies between \(n/4\) and \(n/2\).

\begin{theorem}[Mubayi~\cite{Mubayi12}]
Fix $\alpha\in(1/2,1)$ and $\delta>0$. Then there exists $n_0$ such that every $n$-vertex graph $G$ with $n>n_0$, at least $\lfloor n^2/4\rfloor+1$ edges, and $b(G)<\alpha n/2$ satisfies
\[
   t(G)> (\alpha(1-\alpha)-\delta)\frac{n^2}{4}.
\]
\end{theorem}

Mubayi also showed that the minimum possible number of triangles changes from quadratic to cubic when the book number drops below \(n/4\). 
Since the theorem of Edwards and Khad\v{z}iivanov--Nikiforov gives the lower bound \(b(G)\geq n/6\), this leaves the remaining range \(n/6\leq b<n/4\). 
Conlon, Fox, and Sudakov~\cite{CFS20} proposed the following conjecture giving an exact answer in this range.

\begin{conjecture}[Conlon, Fox, and Sudakov~\cite{CFS20}, see Conjecture 1.1]\label{MainConj}
If $n/6\le b<n/4$, then every graph on $n$ vertices with at least $\lfloor n^2/4\rfloor$ edges and book number at most $b$ which is not the balanced complete bipartite graph has at least $b^2(n-4b)$ triangles, with equality if and only if the graph is $S_{b,n}$.
\end{conjecture}

Here, the graph $S_{b,n}$ denotes the following blow-up of the $3$-prism. Its vertex set is the disjoint union of six parts $U_1,U_2,U_3,V_1,V_2,V_3$, where
\[
 |U_1|=|U_2|=|V_1|=|V_2|=b,
 \qquad |U_3|=\left\lfloor\frac{n-4b}{2}\right\rfloor,
 \qquad |V_3|=\left\lceil\frac{n-4b}{2}\right\rceil .
\]
For $i\ne j$, the pair $U_i,U_j$ is complete and the pair $V_i,V_j$ is complete; for each $i$, the pair $U_i,V_i$ is complete; and there are no other edges. When $n/6\le b\le n/4$, this graph has $n$ vertices, $\lfloor n^2/4\rfloor$ edges, book number $b$, and exactly $b^2(n-4b)$ triangles.

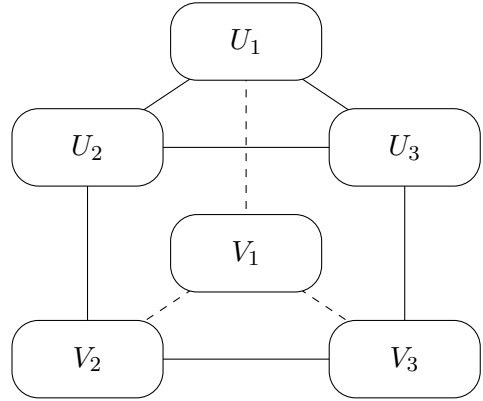
\begin{figure}[htbp]
    \centering
    \begin{tikzpicture}[
        part/.style={draw=black, rounded corners=10pt, inner sep=10pt, minimum width=2cm, minimum height=1cm, fill=white}, scale=0.7
    ]
    \coordinate (U1) at (0, 2);
    \coordinate (U2) at (-3, 0);
    \coordinate (U3) at (3, 0);
    \coordinate (V1) at (0, -2);
    \coordinate (V2) at (-3, -4);
    \coordinate (V3) at (3, -4);
    
    \draw[black, dashed] (U1) -- (V1);
    \draw[black] (U2) -- (V2);
    \draw[black] (U3) -- (V3);
    \draw[black] (U1) -- (U2);
    \draw[black] (U2) -- (U3);
    \draw[black] (U3) -- (U1);
    \draw[black, dashed] (V1) -- (V2);
    \draw[black] (V2) -- (V3);
    \draw[black, dashed] (V1) -- (V3);

    \node[part] at (U1) {$U_1$};
    \node[part] at (U2) {$U_2$};
    \node[part] at (U3) {$U_3$};
    \node[part] at (V1) {$V_1$};
    \node[part] at (V2) {$V_2$};
    \node[part] at (V3) {$V_3$};
    
    \end{tikzpicture}
    \caption{The extremal graph $S_{b,n}$. Each solid or dashed line represents a complete bipartite graph.}
    \label{Fig-S_bn}
\end{figure}

Conlon, Fox, and Sudakov~\cite{CFS20} proved Conjecture~\ref{MainConj} when $(1/4-1/2000)n\le b<n/4$ and when $b=n/6$. 
They also proved a quantitative lower bound near the endpoint $n/6$: 
let $G$ be an $n$-vertex graph with at least $n^2/4$ edges which is not the balanced complete bipartite graph.
If $\eta>0$ is sufficiently small and $b(G)\le (1/6+\eta)n$, then $G$ has at least $\big(\frac{1}{108}-O(\eta^{1/3})\big)n^3$ triangles. 
This has the correct leading term, but it is weaker than the conjectured value in the case $b=(1/6+\eta)n$, namely
\[
   b^2(n-4b)=\left(\frac{1}{108}-\eta^2-4\eta^3\right)n^3,
\]
because the error term $O(\eta^{1/3})$ is much larger than the conjectural loss $\eta^2+4\eta^3$.
In the concluding remarks of~\cite{CFS20}, the authors asked whether ``there is some \(\varepsilon>0\) such that the conjecture holds for all \(n/6\leq b\leq (\frac16+\varepsilon)n\).'' 

Our main result gives an affirmative answer to the above question.

\begin{theorem}\label{MainTheorem}
There exists an absolute constant \(\varepsilon>0\) such that Conjecture~\ref{MainConj} holds for all \(n/6\leq b\leq (1/6+\varepsilon)n\).
\end{theorem}

The rest of the paper is organized as follows. In Section~2, we introduce notation and collect several auxiliary results used throughout the proof. Section~3 gives an overview of the argument and explains the main steps in the proof of Theorem~\ref{MainTheorem}. In Section~4, we establish a stability theorem showing that every extremal graph is close to a blow-up of the \(3\)-prism. In Section~5, we use the exceptional vertices to derive several key parameter inequalities. Finally, in Section~6, we complete the proof of Theorem~\ref{MainTheorem} by showing that these inequalities force the unique extremal configuration \(S_{b,n}\).

\section{Preliminaries}

For a graph $G$, a subset $S \subseteq V(G)$, and vertices $x, y \in V(G)$, we write $N_G(x)$ for the neighborhood of $x$ and $d_G(x) = |N_G(x)|$ for its degree. Let $E_G(S)$ be the set of edges with both ends in $S$ and $e_G(S) = |E_G(S)|$. The set of common neighbors of $x$ and $y$ is denoted by $N_G(x, y)$, and $d_G(x, y) = |N_G(x, y)|$. When the host graph $G$ is clear from the context, we omit the subscript. For a given set $S$, we write $N_S(x)$ for the set of neighbors of $x$ within $S$ and $d_S(x) = |N_S(x)|$. We write $N_S(x, y)$ for the set of common neighbors of $x$ and $y$ inside $S$ and $d_S(x, y)=|N_{S}(x,y)|$.

Given a graph $H$, if a graph $G$ can be obtained by replacing each vertex \(v \in V(H)\) with an independent set \(I_v\) and replacing each edge \(uv \in E(H)\) with a complete bipartite graph between \(I_u\) and \(I_v\), then $G$ is called a {\it blow-up} of \(H\). Note that the graph $S_{b,n}$ defined in the introduction is a blow-up of the $3$-prism.

Let $f, g: \mathbb{Z}^+ \to \mathbb{R}$ be two functions. We write $f = O(g)$ if there exists a constant $C$ such that $|f(n)| \le C g(n)$ for all $n \in \mathbb{Z}^+$.
For any $6$-dimensional real vector $\mathbf{x}=(x_i)_{1\le i\le 6}$, we write $\|\mathbf{x}\|_1\coloneqq \sum_{i=1}^6 x_i$.
For a positive integer \(k\), we write \([k]\) for the set \(\{1,2,\ldots,k\}\).

The following Rearrangement Inequality (see e.g.~\cite{HLP}) is used several times in our later proof.

\begin{lemma}[Rearrangement Inequality]
    Let $a_1,a_2,a_3,a_4,a_5,a_6$ be real numbers such that $a_1\ge a_2\ge a_3$.
    Let $a_4',a_5',a_6'$ be a permutation of $a_4,a_5,a_6$ such that $a_4'\ge a_5'\ge a_6'$. Then,
    $$a_1a_4'+a_2a_5'+a_3a_6'\ge a_1a_4+a_2a_5+a_3a_6,$$
    with equality if and only if $(a_i-a_j)(a_{i+3}-a_{j+3})\ge 0$ for all $1\le i< j\le 3$.
\end{lemma}

Given a graph $G$, denote by $k_4(G)$ the number of $4$-cliques in $G$ and by $k_4^{(3)}(G)$ the number of induced subgraphs in $G$ that are isomorphic to a triangle together with an isolated vertex. The following result of Bollob\'{a}s and Nikiforov \cite{BN05} will be needed in our proof, which provides a quantitative relation between $t(G)$, $b(G)$, and $k_4(G)$.  

\begin{theorem}[Bollob\'{a}s and Nikiforov~\cite{BN05}]\label{THM:BN}
    Let $G$ be an $n$-vertex graph. Then, 
    \[
    \left( 6t(G)-\sum_{v\in V(G)}d(v)^{2}+n\cdot e(G) \right)\cdot b(G)\ge n\cdot t(G)+8k_{4}(G)+2k_{4}^{(3)}(G)
    \]
\end{theorem}

\section{Proof overview}\label{sec:overview}

In the rest of the paper, we prove Theorem~\ref{MainTheorem}. 
Throughout, we fix a copy $P$ of the $3$-prism with vertex set $[6]$ and edge set
\[
\{12,23,13,14,25,36,45,56,46\}.
\]
The case $b=n/6$ follows from Conlon, Fox, and Sudakov~\cite{CFS20}, so we shall assume throughout the proof that
\[
   \frac n6<b\le \left(\frac16+\varepsilon\right)n,
\]
where $\varepsilon>0$ is a sufficiently small absolute constant. Let $G$ be an $n$-vertex graph with
\[
   e(G)\ge \left\lfloor \frac{n^2}{4}\right\rfloor,\qquad b(G)\le b,
   \qquad t(G)\le b^2(n-4b),
\]
which is not the balanced complete bipartite graph. To prove the theorem, it is enough to show that $G$ is isomorphic to $S_{b,n}$. We argue by contradiction and assume that $G\not\cong S_{b,n}$.

The proof has three main parts. First, we prove a stability statement: after deleting a small exceptional set $R$, the remaining graph is a spanning subgraph of a blow-up of the $3$-prism, with all six parts close to size $n/6$ and with almost complete bipartite graphs along the edges of the prism. This is Lemma~\ref{Lemma:good-structure}, proved in Section~\ref{sec:structure}. Second, we show that vertices of $R$ with sufficiently large degree can be assigned to the six prism parts without creating forbidden adjacencies. This reduces the extremal problem to three inequalities involving the resulting part sizes; see Lemma~\ref{count} in Section~\ref{sec:parameters}. Finally, Section~\ref{sec:analysis} proves that these inequalities are incompatible unless the six parameters are exactly those of $S_{b,n}$, contradicting the standing assumption.

\section{A coarse blow-up structure}\label{sec:structure}

Our goal in this section is to prove the following lemma.
Recall that $V(P)=\{1,2,...,6\}$.

\begin{lemma}\label{Lemma:good-structure}
    There exists a subset $R \subseteq V(G)$ of size $O(\varepsilon^{1/12}n)$ such that $G \setminus R$ is a spanning subgraph of a blow-up of $P$. 
    Moreover, let $W_i \subseteq V(G)$ denote the corresponding vertex subsets for $i \in V(P)$. Then we have $\bigl| |W_i| - n/6\bigr| = O(\varepsilon^{1/12}n)$ for every $i\in V(P)$ and $d_{W_j}(x) \ge \bigl(1/6 - O(\varepsilon^{1/12})\bigr)n$ for every $ij \in E(P)$ and $x \in W_i$. 
\end{lemma}
Before the proof, we need several auxiliary lemmas.  Let 
    \[
    R_{0}\coloneqq \left\{v\in V(G)\colon \left|d(v)-\frac{n}{2} \right| \ge \varepsilon^{1/3}n \right\}.
    \]
\begin{lemma}\label{Lemma:regular+good_triangle}
    We have $|R_{0}|=O(\varepsilon^{1/3}n)$. Moreover, all but at most $O(\varepsilon^{1/3}n^{3})$ triangles $xyz$ satisfy the property that $ e(N(x,y)), e(N(y,z)) $, and $ e(N(x,z)) $ are all at most $\varepsilon^{1/3} n^{2}$. 
\end{lemma}
\begin{proof}
    Let $T$ be the set of all triangles in $G$ that do not satisfy the desired property. Then for every $xyz\in T$, we have, say, $e(N(x,y))>\varepsilon^{1/3} n^{2}$, thus there are at least $\varepsilon^{1/3} n^{2}$ edges that form a $K_{4}$ with $xy$. Since each $K_{4}$ is counted in this way at most $6n$ times, we have $k_{4}(G)\ge |T|\cdot \varepsilon^{1/3} n/6$. 
    
    Thus, by Theorem~\ref{THM:BN}, the definition of $R_0$, and the assumption that $e(G)\ge \left\lfloor n^{2}/4 \right\rfloor\ge (n^{2}-1)/4$, we have 
        \begin{align*}
            (6b-n)t(G)& \ge b\sum_{v\in V(G)}d(v)^{2}-bn\cdot e(G)+8k_{4}(G) \ \\
            & =b\sum_{v\in V(G)} \left( d(v)-\frac{n}{2} \right)^{2}+bn\cdot e(G)-\frac{bn^{3}}{4}+8k_{4}(G) \\
            & \ge b|R_{0}|\varepsilon^{2/3}n^{2}-\frac{bn}{4}+4|T|\varepsilon^{1/3}n/3. 
        \end{align*}
    Note that the assumption $\frac{n}{6} < b \le \bigl(\frac{1}{6} + \varepsilon\bigr)n$ implies  $\varepsilon n\ge 1/6$, so $bn/4=O(n^2)\le O(\varepsilon n^4)$. It follows that
    \begin{align}\label{first1}
        (6b-n)t(G)\ge  b|R_{0}|\varepsilon^{2/3} n^{2}+4|T|\varepsilon^{1/3}n/3- O(\varepsilon n^4).
    \end{align}
    On the other hand, since $b\le (\frac{1}{6}+\varepsilon )n$ and $ t(G)\le b^{2}(n-4b)\le  n^{3}/108$, we have $(6b-n)t(G)\le O(\varepsilon n^{4})$. Combining it with inequality \eqref{first1}, we derive that $|R_{0}|=O(\varepsilon^{1/3}n)$ and $|T|=O(\varepsilon^{1/3}n^3)$. This completes the proof of Lemma~\ref{Lemma:regular+good_triangle}. 
\end{proof}

The following Lemma analyzes the (common) degree distribution of vertices relative to the neighborhood partition induced by any vertex outside $R_0$.

\begin{lemma}\label{Lemma:inside_neighborhood}
    Fix $u\in V(G)\setminus R_{0}$ and set $U:=N(u)\setminus R_{0}$ and $V:=V(G)\setminus (U\cup R_{0})$. Then 
    \begin{align}\label{pp1}
        \left|d_{U}(v)-\frac{n}{6}\right|=O(\varepsilon^{1/3}n) \ \text{ and }\ \left|d_{V}(v)-\frac{n}{3}\right|=O(\varepsilon^{1/3}n),
    \end{align}
        for every $v\in U$ with $d_{U}(v)>0$; and
    \begin{align}\label{pp2}
                d_{U}(v,w) = O(\varepsilon^{1/3}n)\ \text{ and }\ \left|d_{V}(v,w)- \frac{n}{6}\right|=O(\varepsilon^{1/3}n),
    \end{align}
    for any edge $vw\in E(U)$.
\end{lemma}
\begin{proof}
    Since $u\notin R_0$ and $|R_0|=O(\varepsilon^{1/3}n)$ (Lemma \ref{Lemma:regular+good_triangle}), we have $\left| |U|-\frac{n}{2}\right|=O(\varepsilon^{1/3}n)$ and $\left| |V|-\frac{n}{2}\right| =O(\varepsilon^{1/3}n)$.
    Since every neighbor of $v$ in $U$ forms a triangle with edge $uv$, we have $d_{U}(v)\le b\le \left(\frac{1}{6}+\varepsilon\right) n$. For the lower bound of $d_{U}(v)$, take $w\in N_{U}(v)$, then similarly $d_{U}(w)\le \left(\frac{1}{6}+\varepsilon\right) n$. It follows that
    \begin{align}\notag
            \left(\frac{1}{6}+\varepsilon\right)n\ge b \ge d_{V}(v,w)&\ge d_{V}(v)+d_{V}(w)-|V| \\ \notag
            &\ge d(v)+d(w)-d_{U}(v)-d_{U}(w)-|V|- 2|R_0| \\ \label{dvw}
            &\ge \frac{n}{2}+\frac{n}{2}-d_{U}(v)-\frac{n}{6}-\frac{n}{2} - O(\varepsilon^{1/3}n), 
        \end{align}
        which implies $d_{U}(v)\ge \left(\frac{1}{6}-O(\varepsilon^{1/3}) \right)n$.  
        Since $|d(v)-\frac{n}{2}|=O(\varepsilon^{1/3}n)$ and $|R_0|=O(\varepsilon^{1/3}n)$, we find
        $\left|d_{V}(v)-\frac{n}{3}\right|=O(\varepsilon^{1/3}n)$.
        Moreover, inequality \eqref{dvw} shows that for every edge $vw\in E(U)$, we have 
         \begin{align*}
            d_{V}(v,w)
            \ge \frac{n}{3} - d_{U}(v) - O(\varepsilon^{1/3}n)\ge \frac{n}{6}-O(\varepsilon^{1/3}n). 
        \end{align*}
        So, $d_{U}(v,w)\le b-d_{V}(v,w)=O(\varepsilon^{1/3}n)$. This completes the proof of Lemma~\ref{Lemma:inside_neighborhood}.
\end{proof}

The following lemma provides the main technical analysis leading to the coarse blow-up structure we require.

\begin{lemma}\label{Lemma:roughly-good-structure}
    There exist subsets $W_{i}\subseteq V(G)$ for $i\in [6]$ such that the following holds. 
    \begin{enumerate}
        \item \label{ppp1} $\left| |W_{i}|-\frac{n}{6}\right|=O(\varepsilon^{1/3}n)$ for $i\in [6]$, and $\max\{|W_{2}\cap W_{3}|,|W_{5}\cap W_{6}|\}=O(\varepsilon^{1/3}n)$;
        \item \label{ppp1.5} $(W_{2}\cup W_{3}\cup W_{4}) \cap (W_{1}\cup W_{5}\cup W_{6}) =W_{1}\cap (W_{5}\cup W_{6})=W_{4}\cap (W_{2}\cup W_{3})=\emptyset$;
        \item \label{ppp2} $\left| d(x)- \frac{n}{2} \right|=O(\varepsilon^{1/3}n)$ for $x\in \cup_{i\in [6]}W_{i}$, $d_{W_{1}\cup W_{5}\cup W_{6} }(x)\ge \left(\frac{1}{2}-O(\varepsilon^{1/3})\right) n$ for $x\in W_{4}$, $d_{W_{1}\cup W_{5}\cup W_{6} }(x)\ge \left(\frac{1}{3}-O(\varepsilon^{1/3})\right) n$ for $x\in W_{2}\cup W_{3}$, $d_{W_{2}\cup W_{3}\cup W_{4} }(x)\ge \left(\frac{1}{2}-O(\varepsilon^{1/3})\right) n$ for $x\in W_{1}$, and $d_{W_{2}\cup W_{3}\cup W_{4} }(x)\ge \left(\frac{1}{3}-O(\varepsilon^{1/3})\right) n$ for $x\in W_{5}\cup W_{6}$;
        \item \label{ppp3} $e(W_{i})=O(\varepsilon^{1/3}n^{2})$ for $i\in [6]$;
        \item \label{ppp4} $e(W_{2}\cup W_{3}\cup W_{4} )\ge \left(\frac{1}{36}-O(\varepsilon^{1/3})\right) n^{2}$ and $e(W_{1}\cup W_{5}\cup W_{6} )\ge \left(\frac{1}{36}-O(\varepsilon^{1/3})\right) n^{2}$.
    \end{enumerate}
\end{lemma}

\begin{proof}
    By Lemma~\ref{Lemma:regular+good_triangle}, we have $|R_{0}|=O(\varepsilon^{1/3}n)$. So there are $O(\varepsilon^{1/3}n^{3})$ triangles containing at least one vertex in $R_{0}$. By the quantitative lower bound of Conlon, Fox, and Sudakov~\cite{CFS20} recalled in the introduction, $t(G)>\frac{n^{3}}{108}-O(\varepsilon^{1/3}n^{3})$. By the latter statement of Lemma~\ref{Lemma:regular+good_triangle}, there exists a triangle $x_{1}x_{2}x_{3}$ in $G\setminus R_{0}$ such that $e(N(x_{1},x_{2})), e(N(x_{2},x_{3}))$, and $e(N(x_{1},x_{3}))$ are all at most $\varepsilon^{1/3} n^{2}$. Set $A\coloneqq N(x_{1})\setminus R_{0}$, $B\coloneqq V(G)\setminus(A\cup R_{0})$, $W_{2}\coloneqq N_{A}(x_{3})$, $W_{3}\coloneqq N_{A}(x_{2})$ and \[
    W_{4}\coloneqq \left\{ x\in A\colon d_{B}(x) \ge \left(\frac{1}{2}-C_{0}\varepsilon^{1/3}\right)n \right\}, 
    \] 
    where $C_{0}$ is a fixed large constant.

\tikzset{every picture/.style={line width=0.75pt}} 
\begin{figure}
\centering

\begin{tikzpicture}[x=0.75pt,y=0.75pt,yscale=-0.7,xscale=0.7]

\draw   (200,77.43) .. controls (200,62.28) and (212.28,50) .. (227.43,50) -- (252.57,50) .. controls (267.72,50) and (280,62.28) .. (280,77.43) -- (280,272.57) .. controls (280,287.72) and (267.72,300) .. (252.57,300) -- (227.43,300) .. controls (212.28,300) and (200,287.72) .. (200,272.57) -- cycle ;
\draw   (380,77.43) .. controls (380,62.28) and (392.28,50) .. (407.43,50) -- (432.57,50) .. controls (447.72,50) and (460,62.28) .. (460,77.43) -- (460,272.57) .. controls (460,287.72) and (447.72,300) .. (432.57,300) -- (407.43,300) .. controls (392.28,300) and (380,287.72) .. (380,272.57) -- cycle ;
\draw    (240,180) -- (420,100) ;
\draw    (240,180) -- (420,250) ;
\draw    (420,100) -- (420,250) ;

\filldraw[fill=white, draw=black] (240,180) circle (2pt);

\filldraw[fill=white, draw=black] (420,100) circle (2pt);

\filldraw[fill=white, draw=black] (420,250) circle (2pt);

\draw   (390,83) .. controls (390,70.3) and (400.3,60) .. (413,60) -- (427,60) .. controls (439.7,60) and (450,70.3) .. (450,83) -- (450,117) .. controls (450,129.7) and (439.7,140) .. (427,140) -- (413,140) .. controls (400.3,140) and (390,129.7) .. (390,117) -- cycle ;
\draw   (392.5,175) .. controls (392.5,163.95) and (404.81,155) .. (420,155) .. controls (435.19,155) and (447.5,163.95) .. (447.5,175) .. controls (447.5,186.05) and (435.19,195) .. (420,195) .. controls (404.81,195) and (392.5,186.05) .. (392.5,175) -- cycle ;
\draw   (390,233) .. controls (390,220.3) and (400.3,210) .. (413,210) -- (427,210) .. controls (439.7,210) and (450,220.3) .. (450,233) -- (450,267) .. controls (450,279.7) and (439.7,290) .. (427,290) -- (413,290) .. controls (400.3,290) and (390,279.7) .. (390,267) -- cycle ;

\draw (415,312.4) node [anchor=north west][inner sep=0.75pt]    {$A$};
\draw (231,312.4) node [anchor=north west][inner sep=0.75pt]    {$B$};
\draw (231,190.4) node [anchor=north west][inner sep=0.75pt]    {$x_{1}$};
\draw (422,103.4) node [anchor=north west][inner sep=0.75pt]    {$x_{2}$};
\draw (422,253.4) node [anchor=north west][inner sep=0.75pt]    {$x_{3}$};
\draw (471,92.4) node [anchor=north west][inner sep=0.75pt]    {$W_{2}$};
\draw (474,242.4) node [anchor=north west][inner sep=0.75pt]    {$W_{3}$};
\draw (471,170.4) node [anchor=north west][inner sep=0.75pt]    {$W_{4}$};

\end{tikzpicture}

\caption{Construction of $W_{2}$, $W_{3}$, and $W_{4}$}
\label{fig:Rough-structure}
\end{figure}
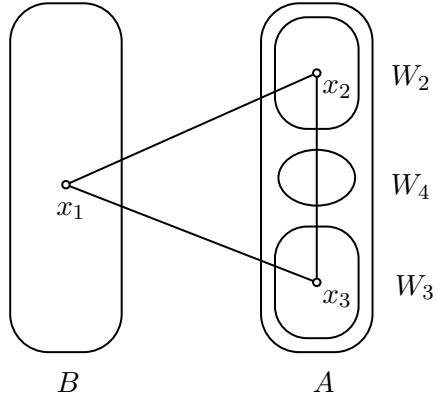

    \begin{claim}\label{Claim:size_of_W234}
        We have
        \begin{enumerate}
            \item \label{pr1} $\left| |A|-\frac{n}{2} \right|=  O(\varepsilon^{1/3}n)$ and $\left| |B|-\frac{n}{2} \right|=  O(\varepsilon^{1/3}n)$;
            \item \label{pr2} $\left| |W_{i}|-\frac{n}{6} \right|=  O(\varepsilon^{1/3}n)$ for $i=2,3$;
            \item \label{pr3} $\left( \frac{3-2\sqrt{2}}{6}-O(\varepsilon^{1/3})\right)n\le |W_{4}|\le \left(\frac{1}{6} + O(\varepsilon^{1/3})\right) n$; 
            \item \label{pr4} $|W_{2}\cap W_{3}|=O(\varepsilon^{1/3}n) \ \text{ and }\ W_{4}\cap (W_{2}\cup W_{3})=\emptyset$;
            \item $\min \{e(A), e(B)\}\ge \left( \frac{1}{36}-O(\varepsilon^{1/3}) \right)n^{2}$. 
        \end{enumerate}
    \end{claim}
    \begin{proof}
        Since $ x_{1}\not\in R_{0} $ and $|R_{0}|=O(\varepsilon^{1/3}n)$ by Lemma~\ref{Lemma:regular+good_triangle}, we derive Property \ref{pr1}.

        Since $|W_{2}|=d_{A}(x_{3})$ and $|W_{3}|=d_{A}(x_{2})$, applying Lemma~\ref{Lemma:inside_neighborhood} with $u=x_{1}, v=x_{3}$ and $u=x_{1}, v=x_{2}$, respectively, we obtain Property \ref{pr2}. Moreover, applying Lemma~\ref{Lemma:inside_neighborhood} with $u=x_{1},v=x_{2}$ and $w=x_{3}$, we have $|W_{2}\cap W_{3}|=|N_{A}(x_{2},x_{3})|=O(\varepsilon^{1/3}n)$. For every $x\in W_{2}\cup W_{3}$, we have $d_{A}(x)>0$ by the definition of $W_{2}$ and $W_{3}$. Applying Lemma~\ref{Lemma:inside_neighborhood} with $u=x_{1}, v=x$ yields $|d_{B}(x)-\frac{n}{3}|=O(\varepsilon^{1/3}n)$. Thus, $x\not\in W_{4}$. So $W_{4}\cap (W_{2}\cup W_{3})=\emptyset$ and, consequently, 
        \[
        |W_{4}|\le |A|-|W_{2}\cup W_{3}|=\left( \frac{1}{6}+ O(\varepsilon^{1/3})\right) n.
        \]
        Note that if $x\in A\setminus W_{4}$, then $d_{A}(x)\ge d(x)-|R_0|-d_{B}(x)>0$, provided $C_0$ is sufficiently large. Applying Lemma~\ref{Lemma:inside_neighborhood} with $u=x_{1}, v=x$ yields $|d_{A}(x)- \frac{n}{6}|=O(\varepsilon^{1/3}n)$. Thus, 
        \begin{align}
            \notag e(A)\ge \frac{1}{2}|A\setminus W_{4}| \cdot & \left(\frac{1}{6}-O(\varepsilon^{1/3}) \right)n \\
            \label{Ineq:lower-bound-edges} &=|A\setminus W_{4}|\cdot \left(\frac{1}{12}-O(\varepsilon^{1/3}) \right)n 
             \ge \left( \frac{1}{36}-O(\varepsilon^{1/3}) \right)n^{2}. 
        \end{align}
        
        By the handshake lemma, we have $\sum_{y\in B}d(y)=2e(B)+e(B,A)+e(B,R_0)$ and $\sum_{x\in A}d(x)=2e(A)+e(B,A)+e(A,R_0)$. 
        Thus,
        \begin{align*}
            \left|e(B)-e(A) \right|
            &\le \frac{1}{2}\left( \left| \sum_{y\in B}d(y)-\sum_{x\in A}d(x) \right| +e(B,R_0)+e(A,R_0) \right) \\
            &\le \frac{1}{2} \left( \left| \sum_{y\in B}d(y)-\sum_{x\in A}d(x) \right|+O(\varepsilon^{1/3}n^{2}) \right). 
        \end{align*}
        For every $x\in A\cup B$, since $x\not\in R_{0}$, we have $\left| d(x)-\frac{n}{2} \right|=O(\varepsilon^{1/3}n)$. Combining these estimates, we obtain 
        \begin{align}\label{eA=eB}
            |e(B)-e(A)|=O(\varepsilon^{1/3}n^{2}). 
        \end{align}
        Together with inequality~\eqref{Ineq:lower-bound-edges}, we have 
        $e(B)\ge \left( \frac{1}{36}-O(\varepsilon^{1/3}) \right)n^{2}$.

        Now, we prove the lower bound of $|W_{4}|$. The key idea is counting triangles in $G$. 
        Suppose $|W_{4}|\le \frac{n}{12}$, as otherwise we are done. First, we count triangles with one vertex in $A$ and two vertices in $B$. 
        Fix an edge $xy\in E(A)$. We call an edge in $E(B)$ {\it good} if it forms a triangle with $x$ or $y$. Applying Lemma~\ref{Lemma:inside_neighborhood} with $u=x_{1}, v=x, w=y$, we have 
        \begin{align*}
            \min\{d_{B}(x), d_{B}(y)\}\ge \left( \frac{1}{3}-O(\varepsilon^{1/3})\right) n\quad \text{and}\quad \left|d_{B}(x,y)- \frac{n}{6}\right|=O(\varepsilon^{1/3}n). 
        \end{align*}
        Therefore, there are at most 
        \[
         \binom{|B|}{2}-\binom{d_{B}(x)}{2}-\binom{d_{B}(y)}{2}+\binom{d_{B}(x,y)}{2}\le \left( \frac{1}{36}+O(\varepsilon^{1/3})\right)n^{2}
        \]
        edges that are not good. Combining it with inequalities~\eqref{eA=eB} and \eqref{Ineq:lower-bound-edges}, we find at least 
        \[
        e(B)-\left( \frac{1}{36}+O(\varepsilon^{1/3})\right)n^{2}\ge e(A)-\left( \frac{1}{36}+O(\varepsilon^{1/3})\right)n^{2}\ge \frac{|A\setminus W_{4}|}{12}n-\frac{n^{2}}{36}-O(\varepsilon^{1/3}n^{2})
        \]
        good edges. Summing these up for every edge in $A$, since $d_{A}(x)\le \left(\frac{1}{6}+O(\varepsilon^{1/3})\right) n$ for every $x\in A$ by Lemma~\ref{Lemma:inside_neighborhood}, the total number of triangles with one vertex in $A$ and two vertices in $B$ is at least 
        \begin{align*}
            &\left(\frac{|A\setminus W_{4}|}{12}n-\frac{n^{2}}{36}-O(\varepsilon^{1/3})n^{2} \right)\cdot e(A)/\left(\frac{1}{6}+O(\varepsilon^{1/3})\right)n \\
            \ge\ &\left(\frac{|A\setminus W_{4}|}{12}n-\frac{n^{2}}{36}-O(\varepsilon^{1/3})n^{2} \right)\cdot |A\setminus W_{4}|\cdot \left(\frac{1}{12}-O(\varepsilon^{1/3}) \right) \cdot n / \left(\frac{1}{6}+O(\varepsilon^{1/3})\right)n\\
            \ge\ &\frac{|A\setminus W_{4}|^{2}}{24}n-\frac{|A\setminus W_{4}|}{72}n^{2}-O(\varepsilon^{1/3}n^{3}), 
        \end{align*}
        where the first inequality follows from inequality~\eqref{Ineq:lower-bound-edges} and that the first item is positive since we have assumed $|W_{4}|\le \frac{n}{12}$. 
        
        Then we count the number of triangles with one vertex in $B$ and two vertices in $A$. For every edge $xy\in E(A)$, applying Lemma~\ref{Lemma:inside_neighborhood} with $u=x_{1}, v=x$ and $w=y$, we have $ d_{B}(x,y) \ge  \left(\frac{1}{6}-O(\varepsilon^{1/3})\right) n$. Thus, the total number of such triangles is at least
        \[
        \sum_{xy\in E(A)}d_{B}(x,y)\ge e(A)\cdot \left(\frac{1}{6}-O(\varepsilon^{1/3})\right) n\ge \frac{|A\setminus W_{4}|}{72}n^{2}-O(\varepsilon^{1/3}n^{3}). 
        \]
        So the number of triangles in $G$ is at least 
        \[
        \frac{|A\setminus W_{4}|^{2}}{24}n-\frac{|A\setminus W_{4}|}{72}n^{2}+ \frac{|A\setminus W_{4}|}{72}n^{2}-O(\varepsilon^{1/3}n^{3}) = \frac{|A\setminus W_{4}|^{2}}{24}n-O(\varepsilon^{1/3}n^{3}). 
        \]
        Recall that $t(G)\le b^{2}(n-4b)\le \frac{n^{3}}{108}$. Hence $|A\setminus W_{4}|\le \left(\frac{\sqrt{2}}{3}+O(\varepsilon^{1/3})\right)n$ and, consequently, $|W_{4}|\ge |A|-|A\setminus W_{4}|\ge \left( \frac{3-2\sqrt{2}}{6}-O(\varepsilon^{1/3})\right)n$. This completes the proof of Claim~\ref{Claim:size_of_W234}.
    \end{proof}
    
    For each vertex $x\in W_{4}$, by the definition of $W_{4}$, there are at most $|B|\cdot (|B|-d_{B}(x))= O(\varepsilon^{1/3}n^{2})$ edges in $B$ that do not form a triangle with $x$. So there are at least $|W_{4}|\cdot \left(e(B)-O(\varepsilon^{1/3}n^{2} ) \right)=\Omega(n^{3})$ triangles with one vertex in $W_{4}$ and two vertices in $B$. By Lemma~\ref{Lemma:regular+good_triangle}, we can choose a triangle $x_{4}x_{5}x_{6}$ with $x_{4}\in W_{4}, x_{5},x_{6}\in B$ such that $e(N(x_{4},x_{5})), e(N(x_{4},x_{6}))$, and $e(N(x_{5},x_{6}))$ are all at most $\varepsilon^{1/3} n^{2}$. 
    
    Set $B' \coloneqq B \cap N(x_{4})$, $W_{5} \coloneqq B' \cap N(x_{6})$, $W_{6} \coloneqq B' \cap N(x_{5})$, and 
    \[
    W_{1}\coloneqq \{y\in B'\colon d_{A}(y)\ge \left(\frac{1}{2}-C_{0}\varepsilon^{1/3}\right)n\}.
    \]

\tikzset{every picture/.style={line width=0.75pt}} 

\begin{figure}
\centering
    
\begin{tikzpicture}[x=0.75pt,y=0.75pt,yscale=-0.7,xscale=0.7]

\draw   (200,77.43) .. controls (200,62.28) and (212.28,50) .. (227.43,50) -- (252.57,50) .. controls (267.72,50) and (280,62.28) .. (280,77.43) -- (280,272.57) .. controls (280,287.72) and (267.72,300) .. (252.57,300) -- (227.43,300) .. controls (212.28,300) and (200,287.72) .. (200,272.57) -- cycle ;
\draw   (380,77.43) .. controls (380,62.28) and (392.28,50) .. (407.43,50) -- (432.57,50) .. controls (447.72,50) and (460,62.28) .. (460,77.43) -- (460,272.57) .. controls (460,287.72) and (447.72,300) .. (432.57,300) -- (407.43,300) .. controls (392.28,300) and (380,287.72) .. (380,272.57) -- cycle ;
\draw   (390,83) .. controls (390,70.3) and (400.3,60) .. (413,60) -- (427,60) .. controls (439.7,60) and (450,70.3) .. (450,83) -- (450,117) .. controls (450,129.7) and (439.7,140) .. (427,140) -- (413,140) .. controls (400.3,140) and (390,129.7) .. (390,117) -- cycle ;
\draw   (392.5,175) .. controls (392.5,163.95) and (404.81,155) .. (420,155) .. controls (435.19,155) and (447.5,163.95) .. (447.5,175) .. controls (447.5,186.05) and (435.19,195) .. (420,195) .. controls (404.81,195) and (392.5,186.05) .. (392.5,175) -- cycle ;
\draw   (390,233) .. controls (390,220.3) and (400.3,210) .. (413,210) -- (427,210) .. controls (439.7,210) and (450,220.3) .. (450,233) -- (450,267) .. controls (450,279.7) and (439.7,290) .. (427,290) -- (413,290) .. controls (400.3,290) and (390,279.7) .. (390,267) -- cycle ;
\draw   (210,80.57) .. controls (210,69.21) and (219.21,60) .. (230.57,60) -- (249.43,60) .. controls (260.79,60) and (270,69.21) .. (270,80.57) -- (270,269.43) .. controls (270,280.79) and (260.79,290) .. (249.43,290) -- (230.57,290) .. controls (219.21,290) and (210,280.79) .. (210,269.43) -- cycle ;
\draw    (240,100) -- (420,175) -- (240,250) -- cycle ;

\filldraw[fill=white, draw=black] (240,100) circle (2pt);

\filldraw[fill=white, draw=black] (420,175) circle (2pt);

\filldraw[fill=white, draw=black] (240,250) circle (2pt);
\draw   (220,85.33) .. controls (220,76.86) and (226.86,70) .. (235.33,70) -- (244.67,70) .. controls (253.14,70) and (260,76.86) .. (260,85.33) -- (260,134.67) .. controls (260,143.14) and (253.14,150) .. (244.67,150) -- (235.33,150) .. controls (226.86,150) and (220,143.14) .. (220,134.67) -- cycle ;
\draw   (220,215.33) .. controls (220,206.86) and (226.86,200) .. (235.33,200) -- (244.67,200) .. controls (253.14,200) and (260,206.86) .. (260,215.33) -- (260,264.67) .. controls (260,273.14) and (253.14,280) .. (244.67,280) -- (235.33,280) .. controls (226.86,280) and (220,273.14) .. (220,264.67) -- cycle ;
\draw   (216.25,175) .. controls (216.25,165.34) and (226.88,157.5) .. (240,157.5) .. controls (253.12,157.5) and (263.75,165.34) .. (263.75,175) .. controls (263.75,184.66) and (253.12,192.5) .. (240,192.5) .. controls (226.88,192.5) and (216.25,184.66) .. (216.25,175) -- cycle ;

\draw    (190,310) -- (218.59,281.41) ;
\draw [shift={(220,280)}, rotate = 135] [color={rgb, 255:red, 0; green, 0; blue, 0 }  ][line width=0.75]    (10.93,-3.29) .. controls (6.95,-1.4) and (3.31,-0.3) .. (0,0) .. controls (3.31,0.3) and (6.95,1.4) .. (10.93,3.29)   ;

\draw (415,312.4) node [anchor=north west][inner sep=0.75pt]    {$A$};
\draw (231,312.4) node [anchor=north west][inner sep=0.75pt]    {$B$};
\draw (471,92.4) node [anchor=north west][inner sep=0.75pt]    {$W_{2}$};
\draw (474,242.4) node [anchor=north west][inner sep=0.75pt]    {$W_{3}$};
\draw (471,170.4) node [anchor=north west][inner sep=0.75pt]    {$W_{4}$};
\draw (421,170.4) node [anchor=north west][inner sep=0.75pt]    {$x_{4}$};
\draw (221,82.4) node [anchor=north west][inner sep=0.75pt]    {$x_{5}$};
\draw (221,252.4) node [anchor=north west][inner sep=0.75pt]    {$x_{6}$};
\draw (164,92.4) node [anchor=north west][inner sep=0.75pt]    {$W_{5}$};
\draw (164,240.4) node [anchor=north west][inner sep=0.75pt]    {$W_{6}$};
\draw (164,170.4) node [anchor=north west][inner sep=0.75pt]    {$W_{1}$};
\draw (171,312.4) node [anchor=north west][inner sep=0.75pt]    {$B'$};

\end{tikzpicture}

\caption{Construction of $W_{1}$, $W_{5}$, and $W_{6}$}
\label{Fig-Construction_of_W156}
\end{figure}
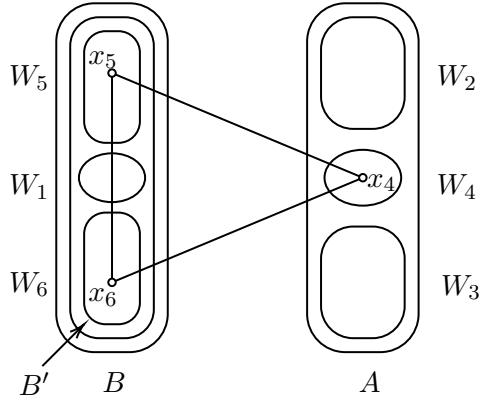
    
    \begin{claim}\label{Claim:size_of_W156}
        We have 
        \begin{enumerate}
            \item $\left| |B'|-\frac{n}{2} \right|=  O(\varepsilon^{1/3}n)$;
            \item $\left| |W_{i}|-\frac{n}{6} \right|=  O(\varepsilon^{1/3}n)$ for $i=5,6$;
            \item $ |W_{1}|\le \left(\frac{1}{6} + O(\varepsilon^{1/3})\right) n $;
            \item $ |W_{5}\cap W_{6}|=O(\varepsilon^{1/3}n) \ \text{ and }\ W_{1}\cap (W_{5}\cup W_{6})=\emptyset $.
        \end{enumerate}
    \end{claim}
    \begin{proof}
        Since $x_{4}\in W_{4}$, we have $|B'|=d_{B}(x_{4})\ge \left(\frac{1}{2}-O(\varepsilon^{1/3})\right)n$. The upper bound follows directly from $|B'|\le |B|\le \left(\frac{1}{2} + O(\varepsilon^{1/3})\right) n$. 
        Note that $d(x_{4})-d_{B'}(x_{4})=d(x_{4})-|B'|=O(\varepsilon^{1/3}n)$. So, we have $$\left| |W_{5}|-d_{N(x_{4})\setminus R_{0}}(x_{6}) \right| = \left|d_{B'}(x_{6})- d_{N(x_{4})\setminus R_{0}}(x_{6}) \right|=O(\varepsilon^{1/3}n),$$ and $$\left| |W_{6}|-d_{N(x_{4})\setminus R_{0}}(x_{5}) \right| = \left|d_{B'}(x_{5})- d_{N(x_{4})\setminus R_{0}}(x_{5}) \right|=O(\varepsilon^{1/3}n).$$ Applying Lemma~\ref{Lemma:inside_neighborhood} with $u=x_{4}, v=x_{6}$ and $u=x_{4}, v=x_{5}$, respectively, we have $|d_{N(x_{4})\setminus R_{0}}(x_{6})-\frac{n}{6}|=O(\varepsilon^{1/3}n)$ and $|d_{N(x_{4})\setminus R_{0}}(x_{5})-\frac{n}{6}|=O(\varepsilon^{1/3}n)$. Thus, $\left| |W_{5}|- \frac{n}{6}\right|=O(\varepsilon^{1/3} n)$ and $\left| |W_{6}|- \frac{n}{6}\right|=O(\varepsilon^{1/3} n)$. Moreover, applying Lemma~\ref{Lemma:inside_neighborhood} with $u=x_{4}, v=x_{5}$ and $w=x_{6}$, we have $|W_{5}\cap W_{6}|\le d_{N(x_{4})\setminus R_{0}}(x_{5},x_{6})=O(\varepsilon^{1/3}n)$. 

        For every $x\in W_{5}\cup W_{6}$, applying Lemma~\ref{Lemma:inside_neighborhood} with $u=x_{4}$ and $v=x$, we have $|d_{A}(x)-\frac{n}{3}|=O(\varepsilon^{1/3} n)$. Thus $x\not\in W_{1}$. So $W_{1}\cap (W_{5}\cup W_{6})=\emptyset$ and, consequently, 
        \[
        |W_{1}|\le |B'|-|W_{5}\cup W_{6}|=\left( \frac{1}{6} +O(\varepsilon^{1/3}) \right)n,
        \]
        This completes the proof of Claim~\ref{Claim:size_of_W156}. 
    \end{proof}

    \begin{claim}\label{Claim:lower_bound_of_W14}
        We have $|W_{1}|\ge \left(\frac{1}{6} - O(\varepsilon^{1/3})\right) n$ and $|W_{4}|\ge \left(\frac{1}{6} - O(\varepsilon^{1/3})\right) n$. 
    \end{claim}
    \begin{proof}
        We count triangles in $G$ in a slightly different way from the proof of Claim~\ref{Claim:size_of_W234}. 

        First, we count the number of triangles with one vertex in $B'$ and two vertices in $A$. Recall that in the proof of Claim~\ref{Claim:size_of_W234} we have shown that the number of triangles with one vertex in $B$ and two vertices in $A$ is at least
        $
        |A\setminus W_{4}|\cdot n^{2}/72-O(\varepsilon^{1/3}n^{3}). 
        $
        Since $|B\setminus B'|=O(\varepsilon^{1/3}n)$, the number of triangles with one vertex in $B'$ and two vertices in $A$ is at least 
        \[
        \frac{|A\setminus W_{4}|}{72}n^{2}-O(\varepsilon^{1/3}n^{3}). 
        \]
        Then we count the number of triangles with one vertex in $A$ and two vertices in $B'$. Note that if $x\in B'\setminus W_{1}$, then $d_{B'}(x)\ge d(x)-d_{A}(x)-O(\varepsilon^{1/3}n)>0$, provided $C_{0}$ is sufficiently large. Applying Lemma~\ref{Lemma:inside_neighborhood} with $u=x_{4}$ and $v=x$ leads to $d_{B'}(x)\ge d_{N(x_{4})\setminus R_{0}}(x)-O(\varepsilon^{1/3}n)\ge \left( \frac{1}{6}-O(\varepsilon^{1/3}) \right)n$. Thus,
        \[
        e(B')\ge \frac{1}{2}|B'\setminus W_{1}|\cdot \left(\frac{1}{6}-O(\varepsilon^{1/3}) \right)n=|B'\setminus W_{1}|\cdot \left(\frac{1}{12}-O(\varepsilon^{1/3}) \right)n. 
        \]
        For each edge $xy\in E(B')$, applying Lemma~\ref{Lemma:inside_neighborhood} with $u=x_{4}, v=x$ and $w=y$, we have $d_{A}(x,y)\ge d_{V(G)\setminus (N(x_{4})\cup R_{0})}(x,y)-O(\varepsilon^{1/3}n)\ge \left( \frac{1}{6}- O(\varepsilon^{1/3})\right)n$. Therefore, the number of triangles with one vertex in $A$ and two vertices in $B'$ is at least 
        \[
        \sum_{xy\in E(B')} d_{A}(x,y)\ge e(B')\cdot \left(\frac{1}{6}-O(\varepsilon^{1/3})\right) n\ge \frac{|B'\setminus W_{1}|}{72}n^{2}-O(\varepsilon^{1/3}n^{3}). 
        \]
        Therefore, we have 
        \[
        \frac{n^{3}}{108}\ge t(G)\ge \frac{|A\setminus W_{4}|}{72}n^{2}+\frac{|B'\setminus W_{1}|}{72}n^{2}-O(\varepsilon^{1/3}n^{3})\ge \frac{n^{3}}{108}-O(\varepsilon^{1/3}n^{3}), 
        \]
        where the last inequality follows by Claim~\ref{Claim:size_of_W234} and Claim~\ref{Claim:size_of_W156}. So the inequality is almost tight and we have 
        \[
        \max\{|A\setminus W_{4}|, |B'\setminus W_{1}|\}\le \left(\frac{1}{3}+O(\varepsilon^{1/3})\right)n, 
        \]
        and the desired result follows. This completes the proof of Claim~\ref{Claim:lower_bound_of_W14}. 
    \end{proof}
    Now we verify that $W_{i}$'s satisfy the desired properties of Lemma \ref{Lemma:roughly-good-structure}. Properties \ref{ppp1} and \ref{ppp1.5} hold by Claim~\ref{Claim:size_of_W234}, Claim~\ref{Claim:size_of_W156}, and Claim~\ref{Claim:lower_bound_of_W14}. 
    
    For every $x\in \cup_{i\in[6]}W_{i}$, since $x\not\in R_{0}$, we have $|d(x)-\frac{n}{2}|=O(\varepsilon^{1/3}n)$. For every $x\in W_{4}$, by the definition of $W_{4}$, we have $d_{W_{1}\cup W_{5}\cup W_{6}}(x)\ge \left( \frac{1}{2}-O(\varepsilon^{1/3})\right)n$. For every vertex $x\in W_{2}\cup W_{3}$, applying Lemma~\ref{Lemma:inside_neighborhood} with $u=x_{1}$ and $v=x$, we have $d_{W_{1}\cup W_{5}\cup W_{6}}(x)\ge \left( \frac{1}{3}-O(\varepsilon^{1/3})\right)n$. For every $x\in W_{1}$, by the definition of $W_{1}$, we have $d_{W_{2}\cup W_{3}\cup W_{4}}(x)\ge \left( \frac{1}{2}-O(\varepsilon^{1/3})\right)n$. For every vertex $x\in W_{5}\cup W_{6}$, applying Lemma~\ref{Lemma:inside_neighborhood} with $u=x_{4}$ and $v=x$, we have $d_{W_{2}\cup W_{3}\cup W_{4}}(x)\ge \left( \frac{1}{3}-O(\varepsilon^{1/3})\right)n$. So Property \ref{ppp2} holds. 

    By the choices of the triangles $x_{1}x_{2}x_{3}$ and $x_{4}x_{5}x_{6}$, we have $ e(W_{i})=O(\varepsilon^{1/3}n^{2})$ for $i\in \{2,3,5,6\}$. Moreover, we have
    \[
    e(W_{1})=\frac{1}{2}\sum_{x\in W_{1}}d_{W_{1}}(x)\le \frac{1}{2}\sum_{x\in W_{1}}(d(x)-d_{W_{2}\cup W_{3}\cup W_{4}}(x))=O(\varepsilon^{1/3}n^{2}). 
    \]
    Similarly, $e(W_{4})=O(\varepsilon^{1/3}n^{2})$. Since $|A\setminus (W_{2}\cup W_{3}\cup W_{4})|=O(\varepsilon^{1/3}n)$, by Claim~\ref{Claim:size_of_W234}, we have $e(W_{2}\cup W_{3}\cup W_{4})\ge \left( \frac{1}{36}-O(\varepsilon^{1/3})\right)n^{2}$. Similarly, by $|B\setminus (W_{1}\cup W_{5} \cup W_{6})|=O(\varepsilon^{1/3}n)$ and Claim~\ref{Claim:size_of_W234}, we have $e(W_{1}\cup W_{5}\cup W_{6})\ge \left( \frac{1}{36}-O(\varepsilon^{1/3})\right)n^{2}$. This completes the proof of Lemma~\ref{Lemma:roughly-good-structure}. 
\end{proof}

Now we are prepared to prove Lemma~\ref{Lemma:good-structure}. 
    
\begin{proof}[Proof of Lemma~\ref{Lemma:good-structure}]
        Let $W_i$ be the subsets that meet the properties in Lemma~\ref{Lemma:roughly-good-structure}, and set $A := W_2 \cup W_3 \cup W_4$ and $B := W_1 \cup W_5 \cup W_6$. 
        Then, Property \ref{ppp1.5} gives $A\cap B=\emptyset$.
        Let 
        \begin{align*}
            B_{1}\coloneqq &\{x\in W_{2}\cup W_{3}\colon d_{W_{5}}(x)\ge \varepsilon^{1/12}n \ \text{and} \ d_{W_{6}}(x)\ge \varepsilon^{1/12}n\} \\
            &\cup \{y\in W_{5}\cup W_{6}\colon d_{W_{2}}(y)\ge \varepsilon^{1/12}n \ \text{and} \ d_{W_{3}}(y)\ge \varepsilon^{1/12}n\}, \\
            B_{2}\coloneqq &\left\{x\in W_{2}\colon d_{W_{3}}(x)\le \left(\frac{1}{6}-\varepsilon^{1/12}\right)n \right\}\cup \left\{x\in W_{3}\colon d_{W_{2}}(x)\le \left(\frac{1}{6}-\varepsilon^{1/12}\right)n \right\} \\
            &\cup \left\{x\in W_{5}\colon d_{W_{6}}(x)\le \left(\frac{1}{6}-\varepsilon^{1/12}\right)n \right\}\cup \left\{x\in W_{6}\colon d_{W_{5}}(x)\le \left(\frac{1}{6}-\varepsilon^{1/12}\right)n \right\}
        \end{align*}
        be the sets of ``bad" vertices. 
    \begin{claim}\label{Claim:few-bad-vertices}
        We have $|B_{1}|\le \varepsilon^{1/12}n$ and  $|B_{2}|\le \varepsilon^{1/12}n$. 
    \end{claim}
    \begin{proof}
        First, we need to do some preparation. For every vertex $x \in W_4$, by Property \ref{ppp2}, we have $d_{A}(x) \le d(x) - d_{B}(x) = O(\varepsilon^{1/3}n)$. Hence, $e(W_4, B) = O(\varepsilon^{1/3}n^2)$. Consequently, by Properties \ref{ppp3} and \ref{ppp4}, 
        \begin{align*}
            \left(\frac{1}{36} - O(\varepsilon^{1/3})\right) n^2 \le e(A) 
             &\le e(W_2, W_3) + e(W_2) + e(W_3) + e(W_4, A) \\
             &= e(W_2, W_3) + O(\varepsilon^{1/3}n^2).
        \end{align*}
        Thus $e(W_2, W_3) \ge \left(\frac{1}{36} - O(\varepsilon^{1/3})\right) n^2$. Since $|W_2|, |W_3| \le \left(\frac{1}{6} + O(\varepsilon^{1/3})\right)n$, we conclude that there are at most $O(\varepsilon^{1/3}n^2)$ missing edges between $W_2$ and $W_3$. Similarly, there are at most $O(\varepsilon^{1/3}n^2)$ missing edges between $W_5$ and $W_6$.
        
        Suppose, for a contradiction, that $|B_{1}|> \varepsilon^{1/12}n$. For every vertex $x\in B_{1}$, we may assume $x\in W_{2}\cup W_{3}$. Then, by the discussion above, there are at least $d_{W_{5}}(x)\cdot d_{W_{6}}(x)-O(\varepsilon^{1/3}n^{2})=\Omega(\varepsilon^{1/6})n^{2}$ edges in $E(W_{5},W_{6})$ that form a triangle with $x$. Hence there are at least $\Omega(\varepsilon^{1/4}n^{3})$ such triangles. For every vertex $x\in W_{4}$, by Property \ref{ppp2}, we have
        \[
        e(N_{B}(x))\ge e(B)-O(\varepsilon^{1/3}n^{2})\ge \left(\frac{1}{36}-O(\varepsilon^{1/3})\right)n^{2},
        \]
        which implies that there are at least $\left(\frac{1}{216}-O(\varepsilon^{1/3})\right)n^{3}$ triangles with one vertex in $W_{4}$ and two vertices in $B$. Similarly, there are at least $\left(\frac{1}{216}-O(\varepsilon^{1/3})\right)n^{3}$ triangles with one vertex in $W_{1}$ and two vertices in $A$. Since $|W_{i}\cap W_{j}|=O(\varepsilon^{1/3}n)$ for $i\neq j\in [6]$, at most $O(\varepsilon^{1/3}n^{3})$ triangles are counted more than once and at most $9$ times. Putting these quantities together and removing duplicate counts, we obtain
        \[
        \frac{n^{3}}{108}\ge t(G)\ge \frac{n^{3}}{108}-O(\varepsilon^{1/3}n^{3})+\Omega(\varepsilon^{1/4}n^{3})=\left(\frac{1}{108}+\Omega(\varepsilon^{1/4})\right)n^{3},
        \]
        which is a contradiction. Therefore, $|B_{1}|\le \varepsilon^{1/12}n$. 

        Suppose that $|B_{2}|\ge \varepsilon^{1/12}n$. 
        Note that every vertex in $B_{2}$ yields at least $\Omega(\varepsilon^{1/12}n)$ missing edges between $W_{2}$ and $W_{3}$ or between $W_{5}$ and $W_{6}$. Hence there are at least $\Omega(\varepsilon^{1/6}n^{2})$ missing edges in total between $W_{2}$ and $W_{3}$ or between $W_{5}$ and $W_{6}$. 
        Recall that there are at most $O(\varepsilon^{1/3}n^{2})$ missing edges between $W_{2}$ and $W_{3}$, and similarly between $W_{5}$ and $W_{6}$. 
        This yields a contradiction. Consequently, $|B_{2}|\le \varepsilon^{1/12}n$.
    \end{proof}
    
    Now set $W_{i}'\coloneqq W_{i}\setminus (B_{1}\cup B_{2})$ for $i\in [6]$ and $R\coloneqq V(G)\setminus (\cup_{i\in [6]}W_{i}')$. Then, $|R|=O(\varepsilon^{1/12}n)$ and $\left| |W_{i}'|-\frac{n}{6}\right| =O(\varepsilon^{1/12}n)$ for $i\in [6]$. We shall show that $R$ and $W_{i}', i\in [6]$ satisfy the properties stated in Lemma~\ref{Lemma:good-structure}. 

    \begin{claim}\label{Claim:degree-condition}
        $G[W_{2}'\cup W_{3}']$ is a bipartite graph with parts $W_{2}'$ and $W_{3}'$. $G[W_{5}'\cup W_{6}']$ is a bipartite graph with parts $W_{5}'$ and $W_{6}'$. And, without loss of generality, we can assume that $d_{W_{j}'}(x)\le \varepsilon^{1/12}n$ for every $x\in W_{i}'$ and $(i,j)\in \{(2,6),(6,2), (3, 5), (5,3)\}$. 
    \end{claim}
    \begin{proof}
        First, we show that $G[W_{2}'\cup W_{3}']$ is a bipartite graph with parts $W_{2}'$ and $W_{3}'$. For any vertex $x\in W_{2}'\cup W_{3}'$, since $x\not\in B_{1}$, we have $d_{W_{5}'}(x)\le \varepsilon^{1/12}n$ or $d_{W_{6}'}(x) \le \varepsilon^{1/12}n$. We say that $x$ has type-$5$ if the first inequality holds and type-$6$ if the second inequality holds. Indeed, $x$ cannot have both types, otherwise $d_{B}(x)\le |W_{1}|+O(\varepsilon^{1/12}n)\le \left(\frac{1}{6}+O(\varepsilon^{1/12})\right)n$, contradicting Property \ref{ppp2}. If two vertices $x,y$ have the same type, say, type-$5$, then by Property \ref{ppp2}, we have 
        \begin{align*}
            d_{W_{1}\cup W_{6}}(x,y) & \ge d_{W_{1}\cup W_{6}}(x)+ d_{W_{1}\cup W_{6}}(y)- |W_{1}\cup W_{6}| \\
            &\ge d_{B}(x)+d_{B}(y)-|W_{1}\cup W_{6}|-O(\varepsilon^{1/12}n) 
            \ge \left(\frac{1}{3}-O(\varepsilon^{1/12})\right)n> b, 
        \end{align*}
        which implies that $x$ and $y$ are not adjacent. So $G[W_{2}'\cup W_{3}']$ is a bipartite graph. For every vertex $x\in W_{2}'$, since $x\not\in B_{2}$, we have $d_{W_{3}'}(x)\ge \left( \frac{1}{6}-O(\varepsilon^{1/12}) \right)n>|W_{3}'|/2 $. So every pair of vertices in $W_{2}'$ has at least one common neighbor. Thus, all vertices in $W_{2}'$ must have the same type. Similarly, all vertices in $W_{3}'$ also must have the same type. Without loss of generality, we may assume that all vertices in $W_{2}'$ have type-$6$ and all vertices in $W_{3}'$ have type-$5$. 

        Similarly, we can assign type-$2$ or type-$3$ to each vertex $x\in W_{5}'\cup W_{6}'$, according to $d_{W_{2}'}(x)\le \varepsilon^{1/12}n$ or $d_{W_{3}'}(x) \le \varepsilon^{1/12}n$, and, by a similar argument, conclude that $G[W_{5}'\cup W_{6}']$ is a bipartite graph with parts $W_{5}'$ and $W_{6}'$. Since every vertex in $W_{2}'$ has type-$6$, we have $e(W_{6}',W_{2}')=\sum_{x\in W_{2}'}d_{W_{6}'}(x)=O(\varepsilon^{1/12}n^{2})$. If all vertices in $W_{6}'$ have type-$3$, then by Property \ref{ppp2},  
        \[
        e(W_{6}',W_{2}')=\sum_{x\in W_{6}'}d_{W_{2}'}(x)\ge \sum_{x\in W_{6}'}\left(d_{A}(x)-|W_{4}'|-O(\varepsilon^{1/12}n)\right)\ge \left( \frac{1}{36}-O(\varepsilon^{1/12})\right)n^{2}, 
        \]
        which is a contradiction. So all vertices in $W_{6}'$ have type-$2$ and all vertices in $W_{5}'$ have type-$3$. This completes the proof of Claim~\ref{Claim:degree-condition}. 
    \end{proof} 

    Now, we verify that $W_{i}'$'s satisfy the desired property from Lemma~\ref{Lemma:good-structure}. Recall that $|W_{i}'|$'s were estimated when they were defined. 
    By Property \ref{ppp1.5} and Claim~\ref{Claim:degree-condition}, the six parts $W_{i}', i\in [6]$, are pair-wise disjoint. By Claim~\ref{Claim:degree-condition}, $G[W_{i}']$ is empty for $i\in\{2,3,5,6\}$.
    Since every pair of vertices in $W_{1}'$ has at least $\left( \frac{1}{2}-O(\varepsilon^{1/3})\right)n> b$ common neighbors in $A$ by Property \ref{ppp2}, $G[W_{1}']$ must be empty. Similarly, $G[W_{4}']$ is also empty. 

    Finally, we show that for any $i\not=j\in [6]$ and $x\in W_{i}'$, $d_{W_{j}'}(x)=0$ whenever $ij\not\in E(P)$ and $d_{W_{j}'}(x)\ge \left( \frac{1}{6}-O(\varepsilon^{1/12})\right)n$ whenever $ij\in E(P)$. By symmetry, we only need to deal with the cases $i=1$ and $i=2$. 
    
    \noindent\underline{\textbf{Case $i=2$:}} For $j=1$, by Claim~\ref{Claim:degree-condition} and Property \ref{ppp2}, we have $d_{W_{1}'}(x)\ge d_{B}(x)-|W_{5}'|-O(\varepsilon^{1/12}n)\ge \left( \frac{1}{6}-O(\varepsilon^{1/12}) \right)n$. The subcase $j=3$ follows by $x\not\in B_{2}$. For $j=4$, for every vertex $y\in W_{4}'$, by Property \ref{ppp2}, $x$ and $y$ have at least $d_{B}(x)+d_{B}(y)-|B|\ge \left( \frac{1}{3}-O(\varepsilon^{1/3}) \right)n$ common neighbors in $B$, thus cannot be adjacent. So, $d_{W_{4}'}(x)=0$. For $j=5$, Claim~\ref{Claim:degree-condition} and Property \ref{ppp2} yields $d_{W_{5}'}(x)\ge d_{B}(x)-|W_{1}'|-O(\varepsilon^{1/12}n)\ge \left( \frac{1}{6}-O(\varepsilon^{1/12}) \right)n$. For $j=6$, take any $y\in W_{6}'$. By Claim~\ref{Claim:degree-condition}, Property \ref{ppp2}, and $x\not\in B_{2}$, we have 
    \begin{align*}
        d_{W_{3}'}(x,y) &\ge d_{W_{3}'}(x)+d_{W_{3}'}(y)-|W_{3}'|\\ 
        &\ge d_{W_{3}'}(x)+d_{A}(y)-|W_{3}'|-|W_{2}'|-|W_{4}'|-O(\varepsilon^{1/12}n)\ge  \left( \frac{1}{6}-O(\varepsilon^{1/12}) \right)n.
    \end{align*}
    Similarly, we have $d_{W_{5}'}(x,y)\ge \left( \frac{1}{6}-O(\varepsilon^{1/12}) \right)n$.
    Thus, $x,y$ have at least $\left( \frac{1}{3}-O(\varepsilon^{1/12}) \right)n$ common neighbors and cannot be adjacent. So, $d_{W_{6}'}(x)=0$. 
    
    \noindent\underline{\textbf{Case $i=1$:}} By Property \ref{ppp2}, we have 
    \[
    d_{W_{j}'}(x)\ge d_{A}(x)-(|A|-|W_{j}'|)\ge \left( \frac{1}{6}-O(\varepsilon^{1/12}) \right)n,
    \]
    for $j= 2,3,4$. By symmetry, the subcases $j=5,6$ follow from subcase $(i,j)=(2,4)$.
    This completes the proof of Lemma~\ref{Lemma:good-structure}.
\end{proof}

\section{From the exceptional set to parameter inequalities}\label{sec:parameters}

Let $W_{i}$ for $i\in [6]$ and $R$ be provided by Lemma~\ref{Lemma:good-structure}. In this section, after making a few adjustments to these subsets, we will combine the initial assumption on $G$ with some new arguments to derive three inequalities involving the sizes of the resulting subsets (see Lemma~\ref{count}). 

First, we divide the exceptional set $R$ into three parts according to the degrees of the vertices as 
\begin{align*}
    R_1 &\coloneqq \left\{x\in R: d_G(x)< \big(\frac{1}{3}+C_{1} \varepsilon^{1/12}\big)n \right\}, \\
    R_2 &\coloneqq \left\{x\in R: \big(\frac{1}{3}+C_{1} \varepsilon^{1/12}\big)n\le d_G(x)< \big(\frac{5}{12}+C_{1} \varepsilon^{1/12}\big)n \right\}, \\ 
    R_3 &\coloneqq \left\{x\in R: d_G(x)\ge \big(\frac{5}{12}+C_{1} \varepsilon^{1/12}\big)n \right\},
\end{align*} 
where $C_{1}$ is a large constant to be determined later. Set $r_i\coloneqq |R_i|$ for each $i\in [3]$, and $r\coloneqq r_1+r_2$.
Then, let $\varphi\colon R_{2}\cup R_{3}\to [6]$ be the function that will be given by Lemma~\ref{Claim:move-back} and take $W'_i\coloneqq W_i\cup \{x\in R_3:\varphi(x)=i \}$ for each $i\in [6]$.
Set $a_i\coloneqq |W'_i|$ for each $i\in [6]$, $\mathbf{a}\coloneqq (a_1,a_2,a_3,a_4,a_5,a_6)$, and $n_1\coloneqq \| \mathbf{a}\|_1$. Without loss of generality, we may assume that 
\begin{align}\label{condition2}
    a_1=\max \{a_i :i\in [6] \} \ {\text{and}}\  a_1\ge a_2\ge a_3.
\end{align}

For any real vector $\mathbf{x}=(x_{i})_{1\le i\le 6}$, the functions $S$, $T$, $F$, $H_1$, and $H_2$ are defined as:
\begin{align*}
    & S(\mathbf{x}) \coloneqq \sum_{ij\in E(P)}x_ix_j, \\
    & T(\mathbf{x}) \coloneqq x_1 x_2 x_3+ x_4 x_5 x_6, \\
    & F(\mathbf{x}) \coloneqq  T(\mathbf{x}) - x_1\left(S(\mathbf{x})-\left\lfloor\frac{\lVert \mathbf{x}\rVert _{1}^{2}}{4}\right\rfloor \right)
        - b^2(\lVert \mathbf{x}\rVert _{1}-4b),\\
    & H_1(\mathbf{x}) \coloneqq S(\mathbf{x})
        - \left\lfloor\frac{\lVert \mathbf{x}\rVert _{1}^{2}}{4}\right\rfloor - \frac{\lVert \mathbf{x} \rVert _{1}r}{12},\\
    & H_2(\mathbf{x})\coloneqq S(\mathbf{x})-x_3(x_1-b) - \left\lfloor\frac{\lVert \mathbf{x}\rVert _{1}^{2}}{4}\right\rfloor - \frac{\lVert \mathbf{x}\rVert _{1}r}{12}.
\end{align*}

Now we state the main result of this section. 

\begin{lemma}\label{count}
    The vector $\mathbf{a}=(a_{i})_{1\le i\le 6}$ and integer $r$ defined above satisfy the following statements.
    \begin{enumerate}
        \item \label{pro1} $r=O(\varepsilon^{1/12}n)$, $|a_{i}-\frac{n}{6}|=O(\varepsilon^{1/12}n)$, $a_{1}=\max\{a_{i}\colon i\in [6]\}$, and $a_{1}\ge a_{2}\ge a_{3}$;
        \item \label{pro2} $\mathbf{a}\ne \left(b,b,\left\lfloor \frac{n-4b}{2}\right\rfloor,b,b,\left\lceil \frac{n-4b}{2}\right\rceil \right)$ and $\mathbf{a}\ne \left(b,b,\left\lceil \frac{n-4b}{2}\right\rceil,b,b,\left\lfloor \frac{n-4b}{2}\right\rfloor \right)$;
        \item \label{pro3} $F(\mathbf{a})\le O(\varepsilon^{1/12}\lVert\mathbf{a}\rVert_{1}^{2}r)$, $H_1(\mathbf{a})\ge -O( \varepsilon^{1/12} \lVert\mathbf{a}\rVert _{1}r)$, and $H_2(\mathbf{a})\ge - O(\varepsilon^{1/12}\lVert\mathbf{a}\rVert _{1}r)$.
    \end{enumerate}
\end{lemma}

To prove Lemma \ref{count}, we need the following lemma, which characterizes the properties of the exceptional set $R$.

\begin{lemma}\label{Claim:move-back}
        There exists a function $\varphi\colon R_{2}\cup R_{3}\to [6]$ such that the following statement holds. 
        For any vertex $x\in R_{2}\cup R_{3}$, we have $d_{W_i}(x)=0$ whenever $i\varphi(x)\notin E(P)$. 
        Furthermore, any two vertices $x,y$ in $R_{3}$ with $\varphi(x)\varphi(y)\notin E(P)$ are not adjacent. 
\end{lemma}
\begin{proof}
    Suppose that $x\in R_{2}\cup R_{3}$.
    If both $d_{W_{1}}(x)$ and $d_{W_{4}}(x)$ are positive, then take two vertices $y_1\in N_{W_{1}}(x)$ and $y_2\in N_{W_{4}}(x)$. 
    By Lemma~\ref{Lemma:good-structure}, we have 
    \begin{align*}
        d_{W_{2}\cup W_{3}\cup W_{4}}(x,y_1)&\ge d_{W_{2}\cup W_{3}\cup W_{4}}(x)+3\times(\frac{1}{6}-O(\varepsilon^{1/12}))n-|W_{2}\cup W_{3}\cup W_{4}|\\ &\ge d_{W_{2}\cup W_{3}\cup W_{4}}(x)-O(\varepsilon^{1/12}n).
    \end{align*}
    Since $d_{W_{2}\cup W_{3}\cup W_{4}}(x,y_1) \le b < \left(\frac{1}{6}+\varepsilon\right)n$, we deduce $d_{W_{2}\cup W_{3}\cup W_{4}}(x)\le (\frac{1}{6}+O(\varepsilon^{1/12}))n$.
    Similarly, analyzing $d_{W_{1}\cup W_{5}\cup W_{6}}(x,y_2)$ yields $d_{W_{1}\cup W_{5}\cup W_{6}}(x)\le (\frac{1}{6}+O(\varepsilon^{1/12}))n$.
    It follows that
    \begin{align*}
        d_{G}(x)\le |R|+d_{W_{2}\cup W_{3}\cup W_{4}}(x)+d_{W_{1}\cup W_{5}\cup W_{6}}(x) 
        &\le \frac{n}{3}+O(\varepsilon^{1/12})n< (\frac{1}{3}+C_{1}\varepsilon^{1/12})n,
    \end{align*}
    provided $C_1$ is sufficiently large, which is a contradiction. So, $d_{W_{1}}(x)=0$ or $d_{W_{4}}(x)=0$. Similarly, we have $d_{W_{2}}(x)d_{W_{5}}(x)=0$ and $d_{W_{3}}(x)d_{W_{6}}(x)=0$. By symmetry, we may assume that $d_{W_{2}}(x)=d_{W_{3}}(x)=0$.
    If $d_{W_{4}}(x)>0$, then provided $C_1$ is sufficiently large, we have
    \[
    d_{G}(x)\le |R|+|W_{4}|+d_{W_{1}\cup W_{5}\cup W_{6}}(x)< (\frac{1}{3}+C_{1}\varepsilon^{1/12})n,
    \]
    a contradiction. So, we have $d_{W_{4}}(x)=0$, and hence we can take $\varphi(x)=1$. 
    Analogously, we can arrange $\varphi(x)$ for each $x\in R_{2}\cup R_{3}$ with $d_{G}(x)\ge (\frac{1}{3}+C_{1}\varepsilon^{1/12})n $. 

    Now, let $x,y$ be two vertices in $R_{3}$ with $ \varphi(x)\varphi(y)\notin E(P)$.
    Then, we have $d_{W_i}(x)=0$ whenever $i\varphi(x)\notin E(P)$, and $d_{W_i}(y)=0$ whenever $i\varphi(y)\notin E(P)$.
    Assume, for a contradiction, that $x$ is adjacent to $y$. Without loss of generality, assume that $\varphi(x)=1$. If $\varphi(y)=1$, then 
    \begin{align*}
        d_G(x,y)&\ge (d_G(x)-|R|)+(d_G(y)-|R|)-|W_2\cup W_3 \cup W_4|\\ &\ge d_G(x) +d_G(y)-\left(\frac{1}{2}+O(\varepsilon^{1/12})\right)n.
    \end{align*}
    Since $d_G(x,y)\le b< (\frac{1}{6}+\varepsilon)n$, we obtain $d_G(x) +d_G(y)\le \left(\frac{2}{3}+O(\varepsilon^{1/12})\right)n$,
    which contradicts the assumption $\min \{d_G(x), d_G(y) \}\ge (\frac{5}{12}+C_{1} \varepsilon^{1/12})n$ provided $C_1$ is sufficiently large.
    It follows that $\varphi(y)\ne 1$, and $\varphi(y)$ equals 5 or 6. Without loss of generality, assume $\varphi(y)=5$. Then, 
    \begin{align*}
        d_G(x,y)&\ge (d_G(x)-|R|-|W_3|)+(d_G(y)-|R|-|W_6|)-|W_2\cup W_4|\\ &\ge d_G(x) +d_G(y)-\left(\frac{2}{3}+O(\varepsilon^{1/12})\right)n.
    \end{align*}
    Again, using $b< (\frac{1}{6}+\varepsilon)n$, we obtain $d_G(x) +d_G(y)\le \left(\frac{5}{6}+O(\varepsilon^{1/12})\right)n$,
    which contradicts $\min \{d_G(x), d_G(y) \}\ge (\frac{5}{12}+C_{1} \varepsilon^{1/12})n$ provided $C_1$ is sufficiently large. This completes the proof of Lemma~\ref{Claim:move-back}.
\end{proof}

Now, let us prove Lemma \ref{count}.

\begin{proof}[Proof of Lemma~\ref{count}]

Lemma \ref{Lemma:good-structure} implies that 
\begin{align}\label{condition1}
    r= O(\varepsilon^{1/12}n),\ 
    \left|a_i - \frac{n}{6}\right| = O(\varepsilon^{1/12}n)\ {\text{for each}}\ i\in [6] ,\ {\text{and}}\ n_1= n - O(\varepsilon^{1/12}n).
\end{align}
Let $Q$ be the blow-up of $P$ such that the vertex $i$ in $P$ is blown up to an independent set of size $a_i$ for each $i\in [6]$.
By Lemma \ref{Claim:move-back}, the induced subgraph $G'\coloneqq G[\cup_{i=1}^6 W'_i]$ of $G$ is a spanning subgraph of $Q$. Note that $|V(G')|=n_1$, $e(Q)=S(\mathbf{a})$, and $t(Q)=T(\mathbf{a})$. 

We claim that $\mathbf{a}\ne \left(b,b,\left\lfloor \frac{n-4b}{2}\right\rfloor,b,b,\left\lceil \frac{n-4b}{2}\right\rceil \right)$ and $\mathbf{a}\ne \left(b,b,\left\lceil \frac{n-4b}{2}\right\rceil,b,b,\left\lfloor \frac{n-4b}{2}\right\rfloor \right)$.
Otherwise, $Q$ is isomorphic to $S_{b,n}$.
Moreover, we have $|V(G)|-|V(G')|=n-n_1=0$, and hence $G=G'$.
Since $G=G'$ is a subgraph of $Q$, we have
$$e(G)\le e(Q)= e(S_{b,n})=\left\lfloor \frac{n^2}{4} \right\rfloor\le e(G),$$
which implies $e(G)= e(Q)$. 
Therefore, $G$ is isomorphic to $S_{b,n}$, which is a contradiction to our assumption.

    Now, it suffices to show that the vector $\mathbf{a}$ and the integer $r$ satisfy Property \ref{pro3} in Lemma~\ref{count}. We first show that $F(\mathbf{a})\le O(\varepsilon^{1/12}\lVert \mathbf{a}\rVert _{1}^{2}r)$. 
    For any edge $e\in E(Q)$, let $t_Q(e)$ be the number of triangles in $Q$ containing $e$.
    Recall that $e(G)\ge \left\lfloor n^2/4 \right\rfloor$. We find
    \begin{align*}
        e(G')\ge e(G)- \sum_{x\in R_1\cup R_2}d_G(x)\ge \left\lfloor \frac{n^2}{4} \right\rfloor - \sum_{x\in R_1\cup R_2}d_G(x).
    \end{align*}
    It follows that
    \begin{align}\label{tri1}\notag
        t(Q)-t(G')\le \sum_{e\in E(Q)\backslash E(G')}t_Q(e) &\le a_1\cdot |E(Q)\backslash E(G')| \\ &\le a_1\left(e(Q)-\left\lfloor\frac{n^2}{4}\right\rfloor + \sum_{x\in R_1\cup R_2}d_G(x) \right).
    \end{align}
    
    For any vertex $x\in R_2$, let $t(x)$ be the number of triangles consisting of $x$ and two vertices in $V(G')$. We claim that there is a constant $C$ such that for any $x\in R_2$, we have $t(x)\ge (d_G(x)-\frac{n}{3})\frac{n}{6}- C\varepsilon^{1/12}n^2$.
    Without loss of generality, assume that $\varphi(x)=1$ and $d_{W'_2}(x)\ge d_{W'_3}(x)$.
    For any $y\in N_{W_2}(x)$, Lemma \ref{Lemma:good-structure} shows that $d_{W'_3}(y)\ge \left( \frac{1}{6}-O(\varepsilon^{1/12}) \right)n$.
    Thus, the number of triangles consisting of $x,y$, and a vertex in $W'_3$ is at least
    $$d_{W'_3}(x)+d_{W'_3}(y)-a_3\ge d_{W'_3}(x)- O(\varepsilon^{1/12}n).$$
    Therefore, we derive
    \begin{align*}
        t(x)\ge d_{W_2}(x) (d_{W'_3}(x)- O(\varepsilon^{1/12}n))&= (d_{W'_2}(x)-O(\varepsilon^{1/12}n)) (d_{W'_3}(x)- O(\varepsilon^{1/12}n))\\
        &= d_{W'_2}(x)d_{W'_3}(x)- O(\varepsilon^{1/12}n^2).
    \end{align*}
    Since $d_{W'_2}(x)+d_{W'_3}(x)\ge d_G(x)-|R|-a_4= d_G(x)-\frac{n}{6}-O(\varepsilon^{1/12}n)$,
    we deduce that
    \begin{align*}
        t(x)&\ge (d_G(x)-\frac{n}{6}-O(\varepsilon^{1/12}n)-a_2)a_2- O(\varepsilon^{1/12}n^2)
        \\ &= (d_G(x)-\frac{n}{3}-O(\varepsilon^{1/12}n))(\frac{n}{6}-O(\varepsilon^{1/12}n))- O(\varepsilon^{1/12}n^2)
        \\ &\ge (d_G(x)-\frac{n}{3})\frac{n}{6}- O(\varepsilon^{1/12}n^2),
    \end{align*}
    as claimed. It follows that
    \begin{align}\label{tri2}
        t(G)-t(G')\ge \sum_{x\in R_2} t(x)\ge \sum_{x\in R_2} (d_G(x)-\frac{n}{3})\frac{n}{6} - O(\varepsilon^{1/12}n^2r_2).
    \end{align}
    Note that the above inequality still holds when $r_2=0$ as the right-hand side would be 0.

    Combining the inequalities \eqref{tri1} and \eqref{tri2}, we obtain that $t(G)-t(Q)$ is at least
    \begin{align*}
        &\sum_{x\in R_2} (d_G(x)-\frac{n}{3})\frac{n}{6}  - a_1\left(e(Q)-\left\lfloor\frac{n^2}{4}\right\rfloor + \sum_{x\in R_1\cup R_2}d_G(x) \right)- O(\varepsilon^{1/12}n^2r_2)\\
        = &\sum_{x\in R_2} d_G(x)(\frac{n}{6}-a_1) -\frac{n^2r_2}{18} - a_1\left(e(Q)-\left\lfloor\frac{n^2}{4}\right\rfloor + \sum_{x\in R_1}d_G(x) \right)- O(\varepsilon^{1/12}n^2r_2).
    \end{align*}
    Since $|a_1-\frac{n}{6}|=O(\varepsilon^{1/12}n)$ and $d_G(x)=O(n)$, we have
    \begin{align*}
        t(G)-t(Q)&\ge -\frac{n^2r_2}{18} - a_1\left(e(Q)-\left\lfloor\frac{n^2}{4}\right\rfloor + \sum_{x\in R_1}d_G(x) \right)- O(\varepsilon^{1/12}n^2r_2).
    \end{align*}
    According to $|a_1-\frac{n}{6}|=O(\varepsilon^{1/12}n)$, $r=r_1+r_2$, and the definition of $R_1$, we find
    \begin{align*}
        t(G)-t(Q)&\ge -\frac{n^2r_2}{18} - a_1\left(e(Q)-\left\lfloor\frac{n^2}{4}\right\rfloor  \right)- a_1(\frac{1}{3}+C_{1} \varepsilon^{1/12})nr_1 - O(\varepsilon^{1/12}n^2r_2)\\
        &= -\frac{n^2r_2}{18} - a_1\left(e(Q)-\left\lfloor\frac{n^2}{4}\right\rfloor  \right)- \frac{n}{6}\times \frac{1}{3}\times nr_1-O(\varepsilon^{1/12}n^2r_1) - O(\varepsilon^{1/12}n^2r_2)\\
        &= -\frac{n^2r}{18} - a_1\left(e(Q)-\left\lfloor\frac{n^2}{4}\right\rfloor  \right)- O(\varepsilon^{1/12}n^2r).
    \end{align*}
    By $n_1= n - O(\varepsilon^{1/12}n)$, $\left|a_1 - \frac{n}{6}\right| = O(\varepsilon^{1/12}n)$, and the fact that
    \begin{align}\label{floor}
        \left\lfloor\frac{n^2}{4}\right\rfloor = \left\lfloor\frac{n_1^2+r^2+2n_1r}{4}\right\rfloor\ge \left\lfloor\frac{n_1^2}{4}\right\rfloor + \frac{n_1r}{2},
    \end{align}
    we derive
    \begin{align*}
        t(G)-t(Q)&\ge -\frac{n^2r}{18} - a_1\left(e(Q)-\left\lfloor\frac{n_1^2}{4}\right\rfloor  \right)+ \frac{a_1n_1r}{2}-O(\varepsilon^{1/12}n^2r)\\
        &= -\frac{n_1^2r}{18} - a_1\left(e(Q)-\left\lfloor\frac{n_1^2}{4}\right\rfloor \right)+ \frac{n_1}{6}\times\frac{n_1r}{2} -O(\varepsilon^{1/12}n^2r)\\
        &= \frac{n_1^2r}{36} - a_1\left(e(Q)-\left\lfloor\frac{n_1^2}{4}\right\rfloor  \right)-O(\varepsilon^{1/12}n_1^2r).
    \end{align*}
    Recall that $t(G)\le b^2(n-4b)$ and $b=\frac{n}{6}+O(\varepsilon n)$. The above inequality yields
    \begin{align*}
        0&\ge t(Q)+\frac{n_1^2r}{36} - a_1\left(e(Q)-\left\lfloor\frac{n_1^2}{4}\right\rfloor  \right)- b^2(n-4b)-  O(\varepsilon^{1/12}n_1^2r)\\
        &=t(Q)- a_1\left(e(Q)-\left\lfloor\frac{n_1^2}{4}\right\rfloor  \right)- b^2(n_1-4b)-  O(\varepsilon^{1/12}n_1^2r).
    \end{align*}
    The desired result then follows by the fact that $t(Q)=T(\mathbf{a})$ and $e(Q)=S(\mathbf{a})$.

    Next, we show that $H_1(\mathbf{a})\ge -O( \varepsilon^{1/12} \lVert \mathbf{a}\rVert _{1}r)$.
    By the assumption that $e(G)\ge \lfloor n^2/4 \rfloor$ and the definition of $R_1$ and $R_2$, we find
    \begin{align*}
        \left\lfloor \frac{n^2}{4} \right\rfloor\le e(G)\le e(G')+\sum_{x\in R_1\cup R_2}d_G(x)\le e(G')+(\frac{5}{12}+C_1 \varepsilon^{1/12})nr.
    \end{align*}
    Together with $n= n_1+r$, $r= O(\varepsilon^{1/12}n)$, and inequality \eqref{floor}, we derive
    \begin{align}\notag
        0&\le e(G')+(\frac{5}{12}+C_1 \varepsilon^{1/12})(n_1+r)r - \left\lfloor\frac{n_1^2}{4}\right\rfloor - \frac{n_1r}{2}+r\\ \label{pre}
        &\le e(G') - \left\lfloor\frac{n_1^2}{4}\right\rfloor - \frac{n_1r}{12}+ O(\varepsilon^{1/12}n_1r).
    \end{align}
    Since $e(G')\le e(Q)=S(\mathbf{a})$, we have $H_1(\mathbf{a})\ge -O(\varepsilon^{1/12} \lVert \mathbf{a}\rVert _{1}r)$.
    
    Finally, we show that $H_2(\mathbf{a})\ge -O( \varepsilon^{1/12} \lVert \mathbf{a}\rVert _{1}r)$. For $1\le i<j\le 3$, let $\overline{e}(W'_i,W'_j)$ be the number of non-edges between $W'_i$ and $W'_j$.
    Let $u$ be any fixed vertex in $W'_2$. For any $v\in N_{W'_3}(u)$, the number of non-neighbors of $v$ in $W'_1$ is at least $d_{W'_1}(u)-b$ (otherwise, $u$ and $v$ will have more than $b$ common neighbors).
    It follows that
    $$\overline{e}(W'_1,W'_3)\ge d_{W'_3}(u)(d_{W'_1}(u)-b).$$
    Summing the above inequality over all $u\in W'_2$ yields
    \begin{align}\label{A1A3}
        \overline{e}(W'_1,W'_3)\ge \frac{1}{a_2}\sum_{u\in W'_2} d_{W'_3}(u)(d_{W'_1}(u)-b).
    \end{align}
    Note that 
    \begin{align}\label{A2}
        \overline{e}(W'_2,W'_1)+\overline{e}(W'_2,W'_3)\ge \sum_{u\in W'_2}(a_1-d_{W'_1}(u)+a_3-d_{W'_3}(u)).
    \end{align}
    Summing inequalities \eqref{A1A3} and \eqref{A2}, we derive 
    \begin{align*}
        \sum_{1\le i<j\le 3}\overline{e}(W'_i,W'_j)&\ge \sum_{u\in W'_2}(\frac{d_{W'_3}(u)(d_{W'_1}(u)-b)}{a_2} +a_1-d_{W'_1}(u)+a_3-d_{W'_3}(u))\\
        & =\sum_{u\in W'_2}(\frac{a_3(d_{W'_1}(u)-b)}{a_2} +a_1-d_{W'_1}(u)+\frac{(a_3-d_{W'_3}(u))(a_2-d_{W'_1}(u)+b)}{a_2})\\ 
        & \ge \sum_{u\in W'_2}(\frac{a_3(d_{W'_1}(u)-b)}{a_2} +a_1-d_{W'_1}(u)+\frac{(a_3-d_{W'_3}(u))(a_2-a_1+b)}{a_2}).
    \end{align*}
    Since $b\ge \frac{n}{6}$ and $|a_i-\frac{n}{6}|=O(\varepsilon^{1/12} n)$ for each $i\in [3]$, we have
    $$\frac{(a_3-d_{W'_3}(u))(a_2-a_1+b)}{a_2}\ge 0.$$
    Thus, we have
    \begin{align}\notag
        \sum_{1\le i<j\le 3}\overline{e}(W'_i,W'_j)
        & \ge \sum_{u\in W'_2}\left(\frac{a_3(d_{W'_1}(u)-b)}{a_2} +a_1-d_{W'_1}(u)\right)\\ \notag
        &=\sum_{u\in W'_2}\left(\frac{a_3(a_1-b)}{a_2} +\frac{(a_1-d_{W'_1}(u))(a_2-a_3)}{a_2}\right)\\ \label{pre2}
        &\ge \sum_{u\in W'_2}\frac{a_3(a_1-b)}{a_2}=a_3(a_1-b).
    \end{align}
    Combining the inequalities \eqref{pre}, \eqref{pre2}, and the fact that 
    $$e(Q)-e(G')\ge \sum_{1\le i<j\le 3}\overline{e}(W'_i,W'_j),$$
    we derive
    \begin{align}\notag
        0\le e(Q)-a_3(a_1-b) - \left\lfloor\frac{n_1^2}{4}\right\rfloor - \frac{n_1r}{12}+ O(\varepsilon^{1/12}n_1r).
    \end{align}
    Then by $e(Q)=S(\mathbf{a})$, we have $H_2(\mathbf{a})\ge -O( \varepsilon^{1/12} \lVert \mathbf{a}\rVert _{1}r)$, proving Lemma~\ref{count}.
\end{proof}

\section{Final parameter analysis}\label{sec:analysis}
Let $r$ and $\mathbf{a}=(a_{i})_{1\le i\le 6}$ be given by Lemma~\ref{count}. For convenience, set $n_{1}\coloneqq \lVert \mathbf{a}\rVert _{1}$. In this section, we focus on the inequalities $F(\mathbf{a})\le O(\varepsilon^{1/12} n_{1}^{2} r)$, $H_1(\mathbf{a})\ge -O(\varepsilon^{1/12} n_{1} r)$, and $H_2(\mathbf{a})\ge -O(\varepsilon^{1/12} n_{1} r)$.
We will first show that, if $a_1\le b$, then 
$F(\mathbf{a})\le O(\varepsilon^{1/12} n_{1}^{2} r)$ and $H_1(\mathbf{a})\ge -O(\varepsilon^{1/12} n_{1} r)$ cannot hold simultaneously unless the vector $\mathbf{a}$ coincides with that of the extremal graph, contradicting Property~\ref{pro2} of Lemma~\ref{count}. 
Assuming $a_1\ge b+1$, we employ an adjustment argument to prove that inequalities $F(\mathbf{a})\le O(\varepsilon^{1/12} n_{1}^{2} r)$ and $H_2(\mathbf{a})\ge -O(\varepsilon^{1/12} n_{1} r)$ are mutually exclusive, which also leads to a contradiction. This will complete our proof of Theorem~\ref{MainTheorem}.

As we mentioned above, we first deal with the case $a_{1}\le b$. 

\begin{lemma}\label{aleb}
    We have $a_1\ge b+1$.
\end{lemma}

\begin{proof}
    For a contradiction, we assume $a_{1}\le b$.
    Let $a_4',a_5',a_6'$ be a permutation of $a_4,a_5,a_6$ such that $a_4'\ge a_5'\ge a_6'$.
    Set $\mathbf{a}'\coloneqq (a_1,a_2,a_3,a_4',a_5',a_6')$.
    Then, $\|\mathbf{a}' \|_1=n_1$ and $T(\mathbf{a}')= T(\mathbf{a})$.
    By the Rearrangement Inequality, we have 
    \begin{align}\label{replace1}
         F(\mathbf{a}')-F(\mathbf{a}) 
        =&\ a_1(a_1a_4+a_2a_5+a_3a_6-a_1a_4'-a_2a_5'-a_3a_6')\le 0,
    \end{align}
    with equality if and only if $(a_i-a_j)(a_{i+3}-a_{j+3})\ge 0$ for all $1\le i< j\le 3$.
    Set 
    \begin{align*}
        a_3''\coloneqq \left\lceil (a_2+a_3+a_4'+a_5'+a_6'-3a_1)/2 \right\rceil=\left\lceil (n_{1}-4a_1)/2 \right\rceil, \\
        a_6''\coloneqq \left\lfloor (a_2+a_3+a_4'+a_5'+a_6'-3a_1)/2 \right\rfloor=\left\lfloor (n_{1}-4a_1)/2 \right\rfloor, 
    \end{align*}
    and $\mathbf{a}''\coloneqq (a_1,a_1,a_3'',a_1,a_1,a_6'')$.
    Then, $\|\mathbf{a}'' \|_1=n_1$.
    Observe that
    \begin{align}\notag
        F(\mathbf{a}'')- F(\mathbf{a}')= &\ a_1 (a_2 - a_1)(a_5' - a_6')
        -a_4' (a_1 - a_5')(a_1 - a_6') 
        -a_1 (a_1 - a_4')(a_1 - a_6') \\ \label{replace4}
        &-a_1 \left\lfloor \frac{(a_3+a_2+ a_5'+a_4'-3a_1-a_6')^2}{4}\right\rfloor
        \\ \label{replace2} \le &\ 0,
    \end{align}
    Note that $T(\mathbf{a}'')=a_1^2(n_1-4a_1)$ and $S(\mathbf{a}'')=\left\lfloor n_1^2/4 \right\rfloor$. Hence,
    \begin{align*}
        F(\mathbf{a}'')=a_1^2(n_1-4a_1)-b^2(n_1-4b)=(b-a_1)(4a_1^2+4b^2+4a_1b-a_1n_1-bn_1).
    \end{align*}
    By the assumption that $b>n/6\ge n_1/6$ and $a_1=\max \{  a_i:i\in [6] \}\ge n_1/6$, we deduce that 
    $4a_1^2+4b^2+4a_1b-a_1n_1-bn_1>0$. 
    Thus,
    \begin{align}\label{replace3}
        F(\mathbf{a}'')\ge 0,
    \end{align}
    with equality if and only if $a_1=b$. 
    Summing the inequalities~\eqref{replace1},~\eqref{replace4},~\eqref{replace3}, and $F(\mathbf{a})\le O(\varepsilon^{1/12} n_{1}^{2} r)$ leads to
    \begin{align*}
        a_1 (a_1 - a_2)(a_5' - a_6')
        &+a_4' (a_1 - a_5')(a_1 - a_6') 
        +a_1 (a_1 - a_4')(a_1 - a_6') \\ 
        &+a_1 \left\lfloor \frac{(a_3+a_2+ a_5'+a_4'-3a_1-a_6')^2}{4}\right\rfloor \le O(\varepsilon^{1/12} n_{1}^{2} r).
    \end{align*}
    Note that every term on the left-hand side of the above inequality is non-negative. So, every term on the left-hand side of the above inequality is at most $O(\varepsilon^{1/12} n_{1}^{2} r)$. Together with the fact that $a_1=\frac{n}{6}+O(\varepsilon^{1/12}n)$ and $a'_4=\frac{n}{6}+O(\varepsilon^{1/12}n)$, we have
    \begin{align}
        (a_1 - a_2)(a_5' - a_6')&= O(\varepsilon^{1/12}n_1r), \label{small1}\\
        (a_1 - a_5')(a_1 - a_6')&= O(\varepsilon^{1/12}n_1r), \label{small2}\\
        a_3+a_2+ a_5'+a_4'-3a_1-a_6'&\le O(\varepsilon^{1/24}n_1^{1/2}r^{1/2})+1. \label{small3}
    \end{align}
    Since $a_1\ge a_4' \ge a_5'\ge a_6'$, inequality \eqref{small2} shows that
    \begin{align}\label{small4}
        a_1-a_4'\le a_1-a_5'\le \sqrt{(a_1 - a_5')(a_1 - a_6')}= O(\varepsilon^{1/24}n_1^{1/2}r^{1/2}).
    \end{align}
    By inequalities \eqref{small3}, \eqref{small4}, and the assumption that $a_2\ge a_3$, we get
    \begin{align*}
        2(a_1-a_2)&= (a_3-a_2)+(a_5'-a_6')-(a_1-a_4')-(a_3+a_2+ a_5'+a_4'-3a_1-a_6')\\
        &\le a_5'-a_6'+ O(\varepsilon^{1/24}n_1^{1/2}r^{1/2}+1).
    \end{align*}
    Substituting the above inequality into inequality~\eqref{small1} yields
    \begin{align*}
        (a_1 - a_2)(2(a_1 - a_2)-O(\varepsilon^{1/24}n_1^{1/2}r^{1/2}+1))\le O(\varepsilon^{1/12}n_1r),
    \end{align*}
    which implies 
    \begin{align}\label{small5}
        a_1 - a_2=O(\varepsilon^{1/24}n_1^{1/2}r^{1/2}+1).
    \end{align}
    It follows that
    \begin{align}\notag
        |a_3-a_6'|&\le (a_1 - a_2)+(a_1 - a_5')+(a_1-a_4')+|a_3+a_2+ a_5'+a_4'-3a_1-a_6'|\\ \label{small6}
        &=O(\varepsilon^{1/24}n_1^{1/2}r^{1/2}+1).
    \end{align}

    Recall from Lemma \ref{count} that $H_1(\mathbf{a})\ge -O(\varepsilon^{1/12} n_{1} r)$. 
    It follows by the Rearrangement Inequality that 
    \begin{align}\notag
        H_1(\mathbf{a}')&=H_1(\mathbf{a})+ 
        a_1a_4'+a_2a_5'+a_3a_6'-a_1a_4-a_2a_5-a_3a_6\\ \label{Ha} &\ge H_1(\mathbf{a})\ge -O(\varepsilon^{1/12} n_{1} r).
    \end{align}
    Observe that
    \begin{align*}
        H_1(\mathbf{a}') 
        =&(a_2-a_1)(a_3-a_6')+(a_1 - a_2)(a_4'-a_1)+(a_1-a_5')(a_3-a_6')+ (a_1-a_5')(a_2-a_1)\\
        &+(a_1-a_4')(a_3-a_6')+(a_1 - a_2)(a_1-a_5')+ (a_1-a_5')(a_1-a_4')\\&- \left\lfloor\frac{(a_1+a_2+a_3-a_4'-a_5'-a_6')^2}{4}\right\rfloor- \frac{n_1r}{12}.
    \end{align*}
    Using inequalities \eqref{small4}, \eqref{small5}, and \eqref{small6}, we derive
    \begin{align*}
        H_1(\mathbf{a}')\le O(\varepsilon^{1/12} n_{1} r)- \frac{n_1r}{12},
    \end{align*}
    contradicting inequality \eqref{Ha} when $r\ge 1$ for sufficiently small $\varepsilon$. Therefore, $r=0$. 

    Now, combining \eqref{replace1}, \eqref{replace2}, \eqref{replace3}, and $F(\mathbf{a})\le O(\varepsilon^{1/12} n_{1}^{2} r)=0$ implies that all these inequalities hold with equality. 
    That is, $(a_i-a_j)(a_{i+3}-a_{j+3})\ge 0$ for all $1\le i< j\le 3$, $(a_2 - a_1)(a_5' - a_6')=0$, $(a_1 - a_5')(a_1 - a_6')=0$, $(a_1 - a_4')(a_1 - a_6')=0$, $|a_3+a_2+ a_5'+a_4'-3a_1-a_6'|\le 1$, and $a_1=b$.
    This yields $a_1= a_2= a_4= a_5=b$ and $|a_3-a_6|\le 1$, which contradicts Property~\ref{pro2} in Lemma~\ref{count}. This completes the proof of Lemma~\ref{aleb}.
\end{proof}

    According to Lemma \ref{aleb}, we have $a_1\ge b+1$. As mentioned at the beginning of this section, we will derive contradiction by a sequence of adjustment lemmas. 
    Let $a_4',a_5',a_6'$ be a permutation of $a_4,a_5,a_6$ such that $a_4'\ge a_5'\ge a_6'$. 
    Set $\mathbf{a}'\coloneqq (a_1,a_2,a_3,a_4',a_5',a_6')$.
    Then, $\|\mathbf{a}' \|_1=n_1$, and the Rearrangement Inequality leads to $F(\mathbf{a}')\le F(\mathbf{a})\le O(\varepsilon^{1/12} n_{1}^{2} r)$
    and $H_2(\mathbf{a}')\ge H_2(\mathbf{a})\ge -O(\varepsilon^{1/12} n_{1} r)$.

    The following lemma shows that we can adjust the value of $a_2$ into $a_1$. 
    
    \begin{lemma}\label{a2a1}
        There exists a vector $\mathbf{a}''=(a''_{i})_{1\le i\le 6}$ satisfying the following properties.
        \begin{itemize}
            \item[1.] $\| \mathbf{a}'' \|_1=n_1$;
            \item[2.] $a''_1 =a''_2 =\max\{a''_i \colon i\in [6]\}\ge b+1$ and $a''_6 = \min \{a''_4 ,a''_5 ,a''_6  \}$;
            \item[3.] $a''_i =\frac{n}{6}+O(\varepsilon^{1/12}n)$ for each $i\in [6]$;
            \item[4.] $F(\mathbf{a}'')\le  O(\varepsilon^{1/12} n_{1}^{2} r)$ and $H_2(\mathbf{a}'')\ge -O(\varepsilon^{1/12} n_{1} r)$.
        \end{itemize}
    \end{lemma}
    \begin{proof}
        We proceed with the following process. Set $\mathbf{a}^{(0)}=(a_{i}^{(0)})_{1\le i\le 6} \coloneqq  \mathbf{a}'$. 
        For any integer $k\ge 0$, if $a_5^{(k)}-a_6^{(k)}\ge a_2^{(k)}-a_3^{(k)}$, then set $\mathbf{a}^{(k+1)}=(a_{i}^{(k+1)})_{1\le i\le 6}$ with coordinates 
        \[
        a_2^{(k+1)}\coloneqq a_2^{(k)}+1, a_3^{(k+1)}\coloneqq a_3^{(k)}-1,\ \text{ and }\ a_i^{(k+1)}\coloneqq a_i^{(k)} \ \text{ for }\ i=1,4,5,6.
        \]
        Otherwise, if $a_5^{(k)}-a_6^{(k)}< a_2^{(k)}-a_3^{(k)}$, then set $\mathbf{a}^{(k+1)}=(a_{i}^{(k+1)})_{1\le i\le 6}$ with coordinates  
        \[
        a_5^{(k+1)}\coloneqq a_5^{(k)}+1, a_6^{(k+1)}\coloneqq a_6^{(k)}-1, \ \text{ and }\ a_i^{(k+1)}\coloneqq a_i^{(k)} \ \text{ for }\ i=1,2,3,4.
        \]
        We continue this process until $a_2^{(k_0)}= a_1$ or $a_5^{(k_0)}=a_1$ for some $k_0\ge 0$.

        Then, by the definition of the process, we have $\| \mathbf{a}^{(k)} \|_1=n_1$, $a_1^{(k)}=a_1=\max\{a_i^{(k)}:i\in [6] \}$, $a_3^{(k)}=\min\{a_i^{(k)}:i\in [3] \}$, and $a_6^{(k)}=\min\{a_i^{(k)}:i=4,5,6 \}$, for each $k\in \{0,1,..., k_0\}$.
        Recall that $a_1,a_2,a_3,a_4',a_5',a_6'$ are all $\frac{n}{6}+O(\varepsilon^{1/12}n)$.
        Note that this process finishes within at most $a_1-a_2+a_1-a_5'=O(\varepsilon^{1/12}n)$ steps.
        So, we have $a_i^{(k)}=\frac{n}{6}+O(\varepsilon^{1/12}n)>0$ for each $i\in [6]$ and $0\le k\le k_0$. 
        We next prove $F(\mathbf{a}^{(k_0)})\le F(\mathbf{a}')$ and $H_2(\mathbf{a}^{(k_0)})\ge H_2(\mathbf{a}')$.

        We first show that, during the process, the value of $F$ does not increase.
        Let $k$ be an integer with $0\le k< k_0$.
        If $a_5^{(k)}-a_6^{(k)}\ge a_2^{(k)}-a_3^{(k)}$, then 
        \begin{align}\label{adjust1}
         F(\mathbf{a}^{(k+1)})-F(\mathbf{a}^{(k)}) 
        =&\ a_1^{(k)}(a_6^{(k)}-a_5^{(k)})\le 0.
        \end{align}
        If $a_5^{(k)}-a_6^{(k)}< a_2^{(k)}-a_3^{(k)}$, then we have
        \begin{align}\notag
        &\ F(\mathbf{a}^{(k+1)})-F(\mathbf{a}^{(k)})\\ \notag
        =&\ a_1^{(k)} - a_1^{(k)} a_2^{(k)} + a_1^{(k)} a_3^{(k)} - a_4^{(k)} + a_1^{(k)} a_5^{(k)} - a_4^{(k)} a_5^{(k)} - a_1^{(k)} a_6^{(k)} + a_4^{(k)} a_6^{(k)}\\ \notag
        =&\ a_1^{(k)} \bigl(1 - a_2^{(k)} + a_3^{(k)} + a_5^{(k)} - a_6^{(k)}\bigr) \;+\; a_4^{(k)} \bigl(-1 - a_5^{(k)} + a_6^{(k)}\bigr)\\ \label{adjust2}
        \le&\ a_1^{(k)}\times 0 + a_4^{(k)}\times (-1)< 0.
        \end{align}
        Therefore, the value of $F$ does not increase, and hence $F(\mathbf{a}^{(k_0)})\le F(\mathbf{a}')\le O(\varepsilon^{1/12}n_1^2r)$.

        Next, we show that the value of $H_2$ does not decrease during the process. Let $k$ be an integer with $0\le k< k_0$.
        Assume that $a_5^{(k)}-a_6^{(k)}\ge a_2^{(k)}-a_3^{(k)}$. 
        By the assumption that $a_1^{(k)}=a_1>b$, we have
        \begin{align}\notag
        H_2(\mathbf{a}^{(k+1)})-H_2(\mathbf{a}^{(k)}) 
        =&\  a_1^{(k)} - a_2^{(k)} + a_3^{(k)} - b + a_5^{(k)} - a_6^{(k)}-1\\ \label{adjust3}
        \ge&\  a_1^{(k)}- b -1\ge 0.
        \end{align}
        Now, assume that $a_5^{(k)}-a_6^{(k)}< a_2^{(k)}-a_3^{(k)}$. Then,
        \begin{align}\label{adjust4}
        H_2(\mathbf{a}^{(k+1)})-H_2(\mathbf{a}^{(k)})
        = a_2^{(k)}-a_3^{(k)}-a_5^{(k)}+a_6^{(k)}-1\ge 0.
        \end{align}
        Thus, the value of $H_2$ does not decrease, and hence $H_2(\mathbf{a}^{(k_0)})\ge H_2(\mathbf{a}')\ge -O(\varepsilon^{1/12}n_1r)$.

    If $a_2^{(k_0)}=a_1$, then taking $\mathbf{a}''\coloneqq \mathbf{a}^{(k_{0})}$ yields the desired result. Now, suppose that $a_2^{(k_0)}\ne a_1$. Then, by the definition of the process, we have $a_5^{(k_0)}=a_1$.
    
    Let us take $\mathbf{a}^{(k_0+1)}= (a_{i}^{(k_{0}+1)})_{1\le i\le 6}$ with coordinates 
    \[
    a_4^{(k_0+1)}\coloneqq a_5^{(k_0)}, a_5^{(k_0+1)}\coloneqq a_4^{(k_0)},\ \text{ and }\ 
    a_i^{(k_0+1)}\coloneqq a_i^{(k_0)} \ \text{ for } \ i=1,2,3,6.
    \]
    Since $a_1^{(k_0)}=a_5^{(k_0)}=a_1\ge \max\{a_2^{(k_0)},a_3^{(k_0)},a_4^{(k_0)},a_6^{(k_0)} \}$, we derive
    \begin{align*}
        &F(\mathbf{a}^{(k_0+1)})-F(\mathbf{a}^{(k_0)})
        =a_1^{(k_0)}(a_1^{(k_0)}-a_2^{(k_0)})(a_4^{(k_0)}-a_5^{(k_0)})\le 0, \qquad \mbox{and}\\
        &H_2(\mathbf{a}^{(k_0+1)})-H_2(\mathbf{a}^{(k_0)})
        = (a_1^{(k_0)}-a_2^{(k_0)})(a_5^{(k_0)}-a_4^{(k_0)})\ge 0.
    \end{align*}
    It follows that $F(\mathbf{a}^{(k_0+1)})\le O(\varepsilon^{1/12}n_1^{2}r)$ and $H_2(\mathbf{a}^{(k_0+1)})\ge -O(\varepsilon^{1/12} n_1r)$.

    Let us take the process stated at the beginning of the proof again: 
    For any integer $k\ge  k_0+1$, if $a_5^{(k)}-a_6^{(k)}\ge a_2^{(k)}-a_3^{(k)}$, then set $\mathbf{a}^{(k+1)}=(a_{i}^{(k+1)})_{1\le i\le 6}$ with the coordinates
    \[
    a_2^{(k+1)}\coloneqq a_2^{(k)}+1, a_3^{(k+1)}\coloneqq a_3^{(k)}-1, \ \text{ and }\  a_i^{(k+1)}\coloneqq a_i^{(k)} \ \text{ for }\ i=1,4,5,6.
    \]
    Otherwise, if $a_5^{(k)}-a_6^{(k)}< a_2^{(k)}-a_3^{(k)}$, then set $\mathbf{a}^{(k+1)}=(a_{i}^{(k+1)})_{1\le i\le 6}$ with the coordinates 
    \[
    a_5^{(k+1)}\coloneqq a_5^{(k)}+1, a_6^{(k+1)}\coloneqq a_6^{(k)}-1, \ \text{ and }\  a_i^{(k+1)}\coloneqq a_i^{(k)} \ \text{ for }\ i=1,2,3,4.
    \]
    We continue this process until $a_2^{(k_1)}= a_1$ or $a_5^{(k_1)}=a_1$ for some $k_1> k_0$.

    Then, we have $\| \mathbf{a}^{(k)} \|_1=n_1$, $a_1^{(k)}=a_4^{(k)}=a_1=\max\{a_i^{(k)}:i\in [6] \}$, $a_3^{(k)}=\min\{a_i^{(k)}:i\in [3] \}$, and $a_6^{(k)}=\min\{a_i^{(k)}:i=4,5,6 \}$, for all $k_0 < k\le k_1$.
    Similarly, since this process finishes within at most $a_1-a_2^{(k_0+1)}+a_1-a_5^{(k_0+1)}=O(\varepsilon^{1/12}n)$ steps,
    we have $a_i^{(k)}=\frac{n}{6}+O(\varepsilon^{1/12}n)>0$ for each $i\in [6]$ and $k_0< k\le k_1$. 
    By the same arguments in inequalities \eqref{adjust1}, \eqref{adjust2}, \eqref{adjust3}, and \eqref{adjust4}, we deduce that the value of $F$ does not increase and the value of $H_2$ does not decrease during the process. Thus, $F(\mathbf{a}^{(k_1)})\le O(\varepsilon^{1/12}n_1^{2}r)$ and $H_2(\mathbf{a}^{(k_1)})\ge -O(\varepsilon^{1/12}n_1r)$.
    If $a_2^{(k_1)}=a_1$, then taking $\mathbf{a}''\coloneqq \mathbf{a}^{(k_1)}$ yields the desired result.
    Now, we suppose that $a_2^{(k_1)}\ne a_1$, and hence $a_5^{(k_1)}=a_1$. 

    By the fact that $\left\lfloor \| \mathbf{a}^{(k_1)} \|_1^2/4\right\rfloor\ge (\| \mathbf{a}^{(k_1)} \|_1^2-1)/4$, we deduce that
    \begin{align*}
        H_2(\mathbf{a}^{(k_1)})\le\ &S(\mathbf{a}^{(k_1)})-a_3^{(k_1)}(a_1-b) - \frac{\| \mathbf{a}^{(k_1)} \|_1^2}{4}+\frac{1}{4} - \frac{n_1r}{12}\\ 
        =\ &- \frac{(a_1-a_2^{(k_1)}-a_6^{(k_1)}+a_3^{(k_1)})^2}{4}-a_3^{(k_1)}(a_1-b) +\frac{1}{4} - \frac{n_1r}{12}\\
        \le\ &-a_3^{(k_1)}(a_1-b)+\frac{1}{4} - \frac{n_1r}{12}.
    \end{align*}
    Together with the assumption that $a_1\ge b+1$ and $a_3^{(k_1)}=\frac{n}{6} +O(\varepsilon^{1/12}n) >\frac{1}{4}$, we obtain $H_2(\mathbf{a}^{(k_1)})< -n_1r/12$, which contradicts $H_2(\mathbf{a}^{(k_1)})\ge -O(\varepsilon^{1/12}n_1r)$ for sufficiently small $\varepsilon$. This completes the proof of Lemma~\ref{a2a1}.
    \end{proof}

    The following lemma shows that we can adjust the values of $a_4''$ and $a_5''$ to be equal.

    \begin{lemma}\label{claim218}
        There exists a vector $\mathbf{a}^* =(a_i^*)_{1\le i\le 6}$ satisfying the following properties. 
        \begin{itemize}
            \item[1.] $\| \mathbf{a}^*  \|_1=n_1$;
            \item[2.] $a_1^* =a_2^*=\max\{a_i^*\colon i\in [6]\}\ge b+1$ and $a_4^*=a_5^*\ge a_6^*$;
            \item[3.] $a_i^*=\frac{n}{6}+ O(\varepsilon^{1/12}n)$ for each $i\in [6]$;
            \item[4.] $F(\mathbf{a}^* )\le O(\varepsilon^{1/12}n_1^2r)$ and $H_2(\mathbf{a}^* )\ge -O(\varepsilon^{1/12} n_1r)$.
        \end{itemize}
    \end{lemma}

    \begin{proof}
    Let $\mathbf{a}''=(a''_i )_{1\le i\le 6}$ be the vector given by Lemma~\ref{a2a1}. Let $\mathbf{a}^* =(a_i^*)_{1\le i\le 6}$ be the vector with coordinates
    \[
    a_4^*=a_5^*=(a''_4 +a''_5 )/2 \ \text{ and }\ a_i^*=a''_i  \ \text{ for } i=1,2,3,6. 
    \]
    We shall show that $\mathbf{a}^* $ satisfies the desired properties. The first three properties follow directly by Lemma~\ref{a2a1}.
    It suffices to confirm the fourth condition.
    Indeed, we have
    \begin{align*}
        F(\mathbf{a}^* )-F(\mathbf{a}'')
        =\frac{1}{4}(a''_4 -a''_5 )^2(a''_6 -a''_1 )\le 0, \quad \mbox{and} \quad
         H_2(\mathbf{a}^* )-H_2(\mathbf{a}'')
        =\frac{1}{4}(a''_4 -a''_5 )^2\ge 0.
    \end{align*}
    This completes the proof of Lemma~\ref{claim218}.
    \end{proof}

Finally, we are ready to present the proof of Theorem~\ref{MainTheorem}.

    \begin{proof}[Proof of Theorem~\ref{MainTheorem}.]
    Let $\mathbf{a}^* $ be a vector given by Lemma~\ref{claim218}. Now, we divide the proof into two cases and derive a contradiction in each. 

    \noindent\underline{\textbf{Case 1:}} We have $a_1^* > a_6^* $ and $a_1^* -2b+a_4^*\le -\dfrac{n_1r}{13(a_1^* -a_6^* )}$.
    
    Set $a^{\dagger}_3 \coloneqq a_3^*+a_4^*+a_5^*-4b+2a_1^* $ and $\mathbf{a}^{\dagger} \coloneqq (a_1^* ,a_2^*,a^{\dagger}_3 ,2b-a_1^* ,2b-a_1^* ,a_6^* )$. 
    Then, $\|\mathbf{a}^{\dagger}  \|_1=\|\mathbf{a}^*  \|_1=n_1$.
    Since $|a_i^*-\frac{n}{6}|=O(\varepsilon^{1/12}n)$ for $i\in [6]$, we have $a^{\dagger}_3 =\frac{n}{6}+O(\varepsilon^{1/12}n)> 0$.
    Recall from Lemma~\ref{claim218} that $a_1^* =a_2^*$ and $a_4^* =a_5^*$. So we have
    \begin{align}\label{typo}
        F(\mathbf{a}^{\dagger} )-F(\mathbf{a}^* )          
        = (a_1^*  - 2 b + a_4^* )(2b+a_4^* -a_1^* )(a_1^* -a_6^* ).
    \end{align}
    Since $|a_1^* -\frac{n}{6}|$, $|a_4^* -\frac{n}{6}|$, and $|b-\frac{n}{6}|$ are all $O(\varepsilon^{1/12}n)$, we have $$2b+a_4^* -a_1^* =\frac{n}{3}+O(\varepsilon^{1/12}n)\ge \frac{n}{4}\ge \frac{n_1}{4},$$
    for sufficiently small $\varepsilon$.
    Together with the assumption of this case, equality~\eqref{typo} implies
    \begin{align}\label{r1}
        F(\mathbf{a}^{\dagger} )\le F(\mathbf{a}^* )-\dfrac{n_1^2r}{52} \le O(\varepsilon^{1/12}n_{1}^{2}r)-\dfrac{n_1^2r}{52}.
    \end{align}
    On the other hand, by the fact that $\lfloor \|\mathbf{a}^{\dagger}  \|_1^2 /4\rfloor\ge (\|\mathbf{a}^{\dagger}  \|_1^2-1)/4$, we derive
    \begin{align*}
        F(\mathbf{a}^{\dagger} )
        \ge&\ (a_1^* )^2 a^{\dagger}_3 + (2b-a_1^* )^2 a_6^* -a_1^* \left( S(\mathbf{a}^{\dagger} ) - \frac{\|\mathbf{a}^{\dagger}  \|_1^2-1}{4}  \right) - b^2(\|\mathbf{a}^{\dagger}  \|_1-4b)\\
        = &\ \frac{1}{4}\left(
        (a_1^* -b)((12 a_6^* - 4 a^{\dagger}_3 )(a_1^* -b)-1)-b+a_1^* (a^{\dagger}_3 -a_6^* )^2\right)\\
        \ge &\ \frac{1}{4}\left(
        (a_1^* -b)((12 a_6^* - 4 a^{\dagger}_3 )(a_1^* -b)-1)-b\right).
    \end{align*}
    Recall that $a_1^* \ge b+1$, $|a_6^* -\frac{n}{6}|=O(\varepsilon^{1/12}n)$, $|a^{\dagger}_3 -\frac{n}{6}|=O(\varepsilon^{1/12}n)$, and $|b-\frac{n}{6}|\le\varepsilon n$. So
    \begin{align}\notag
        F(\mathbf{a}^{\dagger} )
        \ge &\ \frac{1}{4}\left(
        (a_1^* -b)\left(\left(\frac{4}{3}n+O(\varepsilon^{1/12}n)\right)(a_1^* -b)-1\right)-b\right)\\ \notag
        \ge &\ \frac{1}{4}\left(
        \frac{4}{3}n+O(\varepsilon^{1/12}n)-1-b\right)\\ \label{r2}
        =&\ \frac{1}{4}\left(
        \frac{7}{6} n +O(\varepsilon^{1/12}n)-1\right)>0.
    \end{align}
    If $r=0$, inequalities~\eqref{r1} and~\eqref{r2} directly contradict each other. If $r>0$, combining them yields $n_1^2r/52< O( \varepsilon^{1/12}n_1^2r)$, again a contradiction. This completes Case 1.
 
    \noindent\underline{\textbf{Case 2:}} We have $a_1^* = a_6^* $ or $a_1^* -2b+a_4^* > -\dfrac{n_1r}{13(a_1^* -a_6^* )}$.

    Set $a^{\dagger}_3 \coloneqq a_3^*+a_4^* +a_5^*-2a_1^* $ and $\mathbf{a}^{\dagger} \coloneqq (a_1^* ,a_2^*,a^{\dagger}_3 ,a_1^* ,a_1^* ,a_6^* )$. 
    Then, $\|\mathbf{a}^{\dagger}  \|_1=\|\mathbf{a}^*  \|_1=n_1$.
    Since $|a_1^* -\frac{n}{6}|$, $|a_3^*-\frac{n}{6}|$, and $|a_4^* -\frac{n}{6}|$ are all $O(\varepsilon^{1/12}n)$, we have $|a^{\dagger}_3 -\frac{n}{6}|=O(\varepsilon^{1/12}n)$.
    Recall from Lemma~\ref{claim218} that $a_1^* =a_2^*$ and $a_4^* =a_5^*$. So, we have
    \begin{align}\label{r3}
        H_2(\mathbf{a}^{\dagger} )-H_2(\mathbf{a}^* )
        = (a_1^* - a_4^* ) (a_1^*  - 2 b + a_4^* ).
    \end{align}
     If $a_1^* = a_6^* $, since $a_1^* \ge a_4^*\ge a_6^* $ from Lemma~\ref{claim218}, then $a_1^* =a_4^* $, and hence the value of equality~\eqref{r3} is zero. Otherwise, if $a_1^* -2b+a_4^* > -\dfrac{n_1r}{13(a_1^* -a_6^* )}$, then
    \begin{align*}
        (a_1^*  - a_4^* ) (a_1^*  - 2 b + a_4^* )\ge -\frac{n_1r(a_1^*  - a_4^* )}{13(a_1^* -a_6^* )}\ge -\frac{n_1r}{13}.
    \end{align*}
    In both cases, equality \eqref{r3} implies
    \begin{align}\label{r4}
        H_2(\mathbf{a}^{\dagger} )\ge H_2(\mathbf{a}^* )-\frac{n_1r}{13}\ge -O(\varepsilon^{1/12} n_1r)-\frac{n_1r}{13}.
    \end{align}
    On the other hand, by the fact that $\lfloor \| \mathbf{a}^{\dagger}  \|_1^2 /4\rfloor\ge (\| \mathbf{a}^{\dagger}  \|_1^2-1)/4$ and $\| \mathbf{a}^{\dagger}  \|_1=n_1$, we have
    \begin{align*}
        H_{2}(\mathbf{a}^{\dagger} )+\frac{n_1 r}{12}
        \le&\ S(\mathbf{a}^{\dagger} )-a^{\dagger}_3 (a_1^* -b)-\frac{\| \mathbf{a}^{\dagger}  \|_1^2-1}{4}\\
        = &\ \frac{1}{4}\left(1 - 4 a_1^*  a^{\dagger}_3   + 4 a^{\dagger}_3  b  - (a^{\dagger}_3 -a_6^* )^2\right)\\
        \le &\ \frac{1}{4}\left(1 - 4 a_1^*  a^{\dagger}_3   + 4 a^{\dagger}_3  b \right).
    \end{align*}
    By the assumption that $a_1^* \ge b+1$ and $a^{\dagger}_3 =\frac{n}{6}+O(\varepsilon^{1/12}n)>\frac{1}{4}$ (for sufficiently small $\varepsilon$), we have
    \begin{align*}
        H_{2}(\mathbf{a}^{\dagger} )+\frac{n_1r}{12}
        \le \frac{1}{4}\left(1 - 4 a^{\dagger}_3 \right)<0.
    \end{align*}
    If $r=0$, the last inequality gives $H_2(\mathbf{a}^{\dagger})<0$, while~\eqref{r4} gives $H_2(\mathbf{a}^{\dagger})\ge 0$, a contradiction. If $r>0$, combining the last inequality with~\eqref{r4} yields
    $$\frac{n_1r}{12}-\frac{n_1r}{13}< O(\varepsilon^{1/12}n_1r),$$
    again a contradiction. This completes Case 2 and establishes Theorem~\ref{MainTheorem}. 
   \end{proof}

\vspace{0.2cm}

\indent{\it Email address}: \texttt{ckz22000259@mail.ustc.edu.cn}
\vspace{0.2cm}

\indent{\it Email address}: \texttt{jiema@ustc.edu.cn}
\vspace{0.2cm}

\indent{\it Email address}: {wth1115060377@mail.ustc.edu.cn}

\end{document}